\documentclass[a4paper,11pt]{article}
\usepackage{amsfonts}
\usepackage{amsthm}
\usepackage{amsmath}
\usepackage{latexsym}
\usepackage{amssymb}
\usepackage[ansinew]{inputenc}
\usepackage{graphicx}
\usepackage{dsfont}
\newcommand{\nui}[1]{N(#1)}

\newtheorem{fed}{Definition}[section]
\newtheorem*{fed*}{Definition}
\newtheorem*{feds*}{Definitions}
\newtheorem{teo}[fed]{Theorem}
\newtheorem*{teo*}{Theorem}

\newtheorem{cor}[fed]{Corollary}
\newtheorem{pro}[fed]{Proposition}
\theoremstyle{definition}
\newtheorem{rem}[fed]{Remark}

\newtheorem*{rems*}{Remarks}

\newtheorem{exa}[fed]{Example}

\newtheorem{nota}[fed]{Notation}

\newtheorem{probs}[fed]{Problems}

\def\coma{\, , \, }

\def\py{\peso{and}}
\newcommand{\peso}[1]{ \quad \text{ #1 } \quad }

\def\n0{n_{ \text{\rm \tiny o}}}

\newcommand{\IN}[1]{\mathbb {I} _{#1}}

\def\suml{\sum\limits}

\def\bce{\begin{center}}
\def\ece{\end{center}}
\newcommand{\trivial}{\{0\}}

\def\py{\peso{and}}
\def\rk{\text{\rm rk}}
\def\noi{\noindent}
\def\cF{\mathcal F}
\def\cG{\mathcal G}

\def\QED{\hfill $\square$}
\def\EOE{\hfill $\triangle$}
\def\EOEP{\tag*{\EOE}}
\def\uno{\mathds{1}}

\def\bm{\left[\begin{array}}
\def\em{\end{array}\right]}
\def\ben{\begin{enumerate}}
\def\een{\end{enumerate}}
\def\bit{\begin{itemize}}
\def\eit{\end{itemize}}
\def\barr{\begin{array}}
\def\earr{\end{array}}
\def\igdef{\ \stackrel{\mbox{\tiny{def}}}{=}\ }

\def\la{\lambda}

\def\N{\mathbb{N}}
\def\R{\mathbb{R}}

\def\C{\mathbb{C}}
\def\I{\mathbb{I}}

\def\T{\mathbb{T}}

\def\G{\mathcal{G}}
\def\cH{\mathcal{H}}

\def\cM{{\cal M}}
\def\cB{{\cal B}}

\def\cV{{\cal V}}
\def\cU{{\cal U}}

\def\ca{\mathbf{a}}

\def\orto{^\perp}
\def\inc{\subseteq}

\def\sii{ if and only if }

\def\api{\langle}
\def\cpi{\rangle}

\def\da{^\downarrow}

 \DeclareMathOperator{\tr}{tr}
\DeclareMathOperator{\gen}{span}

\DeclareMathOperator{\leqp}{\leqslant}

\newcommand{\mat}{\mathcal{M}_d(\mathbb{C})}

\newcommand{\matsad}{\mathcal{H}(d)}

\newcommand{\matud}{\mathcal{U}(d)}

\newcommand{\matpos}{\mat^+}

\newcommand{\matrec}[1]{\mathcal{M}_{#1} (\mathbb{C})}

\def\beq{\begin{equation}}
\def\eeq{\end{equation}}

\def\pausa{\medskip\noi}

\def\Ax2{\,( S_{E(\cF)^\#_\cV})\hat{}_x }

\newcommand{\Id}{\I_d}
\newcommand{\tcal}{\T_{d}(\ca)}

\newcommand{\asubi}{\ca=(a_i)_{i\in\I_k}\,\in(\R^k_{>0})\da}
\newcommand{\asinsubi}{\ca \in(\R^k_{>0})\da}

\newcommand{\tnsa}{\Theta_{( N \coma  S\coma \ca)}}

\newcommand{\then}{\Rightarrow}
\newcommand{\pint}[1]{\langle #1 \rangle}

\newcommand{\corch}[1]{\left[ #1 \right]}

\newcommand{\llav}[1]{\left\{#1\right\}}

\newcommand{\norm}[1]{\left\|#1\right\|}

\def\nug0{\nu(\cG_0)}

\usepackage{color}

\definecolor{rojo}{rgb}{1,0,0}
\definecolor{azul}{rgb}{0,0,1}

\begin{document}

\title{Generalized frame operator distance problems}
\author{ Pedro G. Massey $^{*}$, Noelia B. Rios $^{*}$ and Demetrio Stojanoff 
\footnote{Partially supported by CONICET 
(PIP 0150/14), FONCyT (PICT 1506/15) and  FCE-UNLP (11X681), Argentina. } \ 
 \footnote{ e-mail addresses: massey@mate.unlp.edu.ar, nbrios@mate.unlp.edu.ar, demetrio@mate.unlp.edu.ar}
\\ {\small CMaLP-FCE-UNLP
and IAM-CONICET, Argentina  }}
\date{}
\maketitle
\begin{abstract}
Let $S\in\matpos$ be a positive semidefinite $d\times d$ complex matrix and let $\ca=(a_i)_{i\in\I_k}\in \R_{>0}^k$, indexed by 
$\I_k=\{1,\ldots,k\}$, be a $k$-tuple of positive numbers. Let $\tcal$ denote the set of families $\cG=\{g_i\}_{i\in\I_k}\in (\C^d)^k$
such that $\|g_i\|^2=a_i$, for $i\in\I_k$; thus, $\tcal$ is the product of spheres in $\C^d$ endowed 
with the product metric. For a strictly convex unitarily invariant norm $N$ in $\mat$, we consider the generalized frame operator distance 
function $\tnsa$ defined on $\tcal$, given by $$\tnsa(\cG)=N(S-S_{\cG}) \peso{where} S_{\cG}=\sum_{i\in\I_k} g_i\,g_i^*\in\matpos\,.$$
In this paper we determine the geometrical and spectral structure of local minimizers $\cG_0\in\tcal$  of $\tnsa$. In particular,
we show that local minimizers are global minimizers, and that these families do not depend on the particular choice of $N$. 
\end{abstract}



\tableofcontents

\section{Introduction}

Matrix approximation problems are ubiquitous in applications of matrix analysis. Following \cite{Hig} these problems can be briefly described
as follows: given $S\in\mat$, a complex matrix of size $d$, 
a matrix norm $N$ in $\mat$, 
and a set $\mathcal X\subset \mat$ then we search for the minimal distance 
$$ 
d_{N}(S,\mathcal X)=\min \{ N(S- A):\ A\in\mathcal X\} \ ,
$$ 
and for the best approximations
of $S$ from $\mathcal X$ (or nearest members in $\mathcal X$) 
$$ 
\mathcal A^{\rm op}_{N}(S\coma \mathcal X) 
=\{A\in\mathcal X:\ N(S-A)= d_{N}(S \coma \mathcal X)\}\ .
$$
Solving these problems, 
that are also known as matrix nearness or Procrustes problems in the literature
(see for example the recent text \cite{GB}, and the classic books of Bhatia \cite{Bhat} and Kato \cite{Ka})  
amounts to 
provide a characterization and, if possible, an explicit computation (in some cases sharp estimations) 
of $d_{N}(S \coma \mathcal X)$ and of the set of best approximations 
$\mathcal A^{\rm op}_{N}(S\coma \mathcal X)$. 
A typical choice for $N$ is the Frobenius norm (also called $2$-norm) since it is an euclidean 
norm (i.e. it is the norm associated with an inner product 
in $\mat$). Still, some other norms are also of interest such as weighted norms, the $p$-norms for $1\leq p$ (that contain the Frobenius norm), or the 
more general class of unitarily invariant norms. Some of the most important choices for $\mathcal X$ are the set of: selfadjoint matrices, 
positive semidefinite matrices, correlation matrices, orthogonal projections, oblique projections, matrices with rank bounded by a fix number 
(see \cite{CMR,Eld,Hig,Ki,Wat}).  

\pausa
Once the nearness problem above has been solved for some $S$, some set $\mathcal X$ and norm $N$ in $\mat$ then,
a natural proximity problem arises: for a fixed $A_0\in\mathcal X$, we search for (some sharp upper bound of) the distance 
$$ 
d_\mathcal X(A_0\coma  \mathcal A^{\rm op}_{N}(S\coma \mathcal X))
=\min\{d_\mathcal X(A_0\coma A):\ A\in \mathcal A^{\rm op}_{N}(S\coma \mathcal X)\}\ ,
$$
where $d_\mathcal X$ denotes a metric in $\mathcal X$. In case $\mathcal X$ can be endowed with a smooth structure that is compatible with
$d_\mathcal X$ and such that $\Psi(A)=N(A_0-A)$ is also a smooth function on $\mathcal X$, then estimations of 
$d_\mathcal X(A_0\coma  \mathcal A^{\rm op}_{N}(S\coma \mathcal X))$ can be obtained by 
applying gradient descent algorithms for $\Psi$ or by studying the evolution of the solutions of flows in $\mathcal X$ associated with the gradient of $\Psi$.

\pausa
Motivated by some optimization problems in finite frame theory, in \cite{dnp} we considered the following matrix nearness problem.
Fix an arbitrary positive semidefinite $S\in\matpos$ and a finite sequence of positive numbers $\ca=(a_i)_{i\in\I_k}\in (\R_{>0})^k$,
 indexed by $\I_k=\{1 \coma \ldots\coma k\}$; we considered the sets 
$$
\tcal=\left\{\{g_i\}_{i\in\I_k}\in (\C^d)^k: \, \|g_i\|^2=a_i \coma i\in\I_k\,\right\} \py
 \mathcal X_\ca=\left\{ \sum_{i\in\I_k} g_i\, g_i^*: \{ g_i\}_{i\in\I_k}\in\tcal \right\}\ .
$$
With this notation we solved the matrix nearness problem corresponding 
to $\mathcal X_\ca\subset \matpos$, for an 
arbitrary strictly convex unitarily invariant norm $N$ in $\mat$. That is, we obtained an explicit description
of $d_{N}(S\coma \mathcal X_\ca)=d_{N}(S\coma \ca)$ and 
$\mathcal A^{\rm op}_{N}(S\coma \mathcal X_\ca)=\mathcal A^{\rm op}_{N}(S\coma \ca)$.
We point out that the set $\mathcal X_\ca$ above can also be described 
as the set of frame operators $S_\mathcal G$ of finite sequences $\cG=\{g_i\}_{i\in\I_k}\in\tcal$ (see Section \ref{sec prelis} for details).

\pausa 
It is then natural to consider the proximity problem associated to the matrix nearness 
problem that we just described. Indeed, because of our initial motivation on this problem, 
we further pose the following (stronger) version: for $\cG_0\in\tcal$ search for a (sharp) upper bound 
of 
$$ 
d(\cG_0\coma \mathcal B^{\rm op}_{N}(S\coma \ca))
=\min\left\{d_{\tcal} (\cG_0\coma  \cG):\ \cG\in \mathcal B^{\rm op}_{N}(S\coma \ca) \right\} \ , 
$$
where 
$$
\mathcal B^{\rm op}_{N}(S\coma \ca)
=\left\{\cG\in\tcal:\ S_\cG\in \mathcal A^{\rm op}_{N}(S\coma \ca)\right\} 
\py
d_{\tcal}^{\, 2}(\cG_0\coma  \cG)=\sum_{i\in\I_k}\|g_i^0-g_i\|^2\ ,
$$ 
for $\cG_0=\{g_i^0\}_{i\in\I_k},\, \cG=\{g_i\}_{i\in\I_k}\in\tcal$.
That is, we shift our attention from frame operators $S_\cG$ to finite sequences $\cG\in\tcal$.
Notice that in the particular case $S=\frac{k}{d}\,I$, $a_i=1$ for $i\in\I_k$ and 
$N$  is the Frobenius norm,
 this problem is related with Paulsen's proximity problem \cite{BoC10,CC13,CFM12}, which is a central open problem in finite frame theory.

\pausa
In case the norm $N$ is sufficiently smooth, we could apply gradient descent algorithms to the function 
$\Theta=\Theta_{(N\coma S\coma\ca)}$ defined on
$\tcal$ - which is a smooth manifold in $(\C^d)^k$ - given by 
$\Theta(\cG)=N(S-S_{\cG})$, starting at $\cG_0$. Such an approach was considered by N. Strawn \cite{strawn,strawn2} for the Frobenius norm $N$.
Also, we could study the evolution of solutions of gradient flows as considered in \cite{LPS}.

\pausa
In the general case, the analysis of the behavior of gradient descent algorithms leads to the study 
the local behavior of the map $$\tcal\ni \cG=\{g_i\}_{i\in\I_k} \mapsto \Theta(\cG)=N(S-S_\cG)\peso{around} \cG_0\in\tcal\,.$$ 
One important issue is determining whether local minimizers of $\Theta$ 
(that are natural attractors of gradient descent algorithms) are actually global minimizers. 
In \cite{dnp} we settled this question in the affirmative for the Frobenius norm (thus solving a conjecture in \cite{strawn}), by relating 
frame operator distance problems in the Frobenius norm with frame completion problems for the Benedetto-Fickus frame potential introduced in \cite{BF}.
Unfortunately, the techniques used in \cite{dnp} do not apply for arbitrary $N$ (not even for $p$-norms with $p>1$, $p\neq 2$).
In the present work we tackle this problem and show 
that, in case $N$ is an arbitrary strictly convex u.i.n., local minimizers of $\Theta$ are characterized by a spectral condition that does not depend on $N$, but only
on $S$ and $\ca$. In particular, we conclude that local minimizers are global minimizers and do not depend on the particular choice of $N$. Our techniques
rely on majorization theory and Lidskii's local theorems for unitarily invariant norms obtained in \cite{dnp2}; indeed, in that paper
we showed that in some particular cases, local minimizers of the generalized frame operator distance (GFOD) functions 
(i.e. $\Theta(\cG)=N(S-S_\cG)$) are global minimizers. Based on the
features of these particular cases, we introduce the notion of co-feasible GFOD problems. Although in general GFOD problems are not co-feasible, this notion plays a crucial role in the study of the spectral structure of local minimizers.
Using that the map $\tcal\ni \cG\mapsto S_\cG\in\mathcal X_\ca$ is continuous, as a byproduct we obtain that local minimizers $S_{\cG_0}\in\mathcal X_\ca$ of the function 
$$ \Psi:\mathcal X_\ca\rightarrow \R_{\geq 0} \peso{given by} \Psi(S_{\cG})= N(S-S_\cG)$$ are global minimizers and do not depend on the choice of 
strictly convex u.i.n. $N$. This last fact is weaker than the result for the functions $\Theta$, since the continuous map $\tcal\ni \cG\mapsto S_\cG\in\mathcal X_\ca$
does not have local cross sections around an arbitrary $\cG_0\in\tcal$.

\pausa
The paper is organized as follows. In Section \ref{sec prelis} we include some preliminary material on matrix analysis and finite frame theory that 
is used throughout the paper. In Section \ref{sec tris} we state our main problem namely, the study of the geometrical and spectral structure of local minimizers of the GFOD functions (i.e. $\Theta$ as above), 
associated to a strictly convex unitarily invariant norm. 
 We begin by  obtaining a series of results related with what we call the inner structure of such local minimizers. 
In section \ref{sec cuatris} we state our main results namely, that local minimizers of GFOD functions are
global minimizers, and give an algorithmic construction of the eigenvalues of such families.
Finally, in Section \ref{Appendixity} we give detailed proofs of some results stated in Section \ref{sec tris}. 

\section{Preliminaries}\label{sec prelis}

In this section we introduce the notation, terminology and results from matrix analysis (see the text \cite{Bhat})
and finite frame theory (see the texts \cite{TaF,FinFram, Chr}) that we will use throughout the 
paper.

\subsection{Matrix Analysis}\label{prel mat an}

\pausa
{\bf Notation and terminology}. We let $\mathcal M_{k,d}(\C)$ be the space of complex $k\times d$ matrices and write $\mathcal M_{d,d}(\C)=\mat$ for the algebra of $d\times d$ complex matrices. We denote by $\matsad\subset \mat$ the real subspace of selfadjoint matrices and by $\matpos\subset \matsad$ the cone of positive 
semi-definite matrices. We let $\matud\subset \mat$ denote the group of unitary matrices.

\pausa
For $d\in\N$, let $\I_d=\{1,\ldots,d\}$. 
Given a vector $x\in\C^d$ we denote by $D_x$ the diagonal matrix in $\mat$ whose main diagonal is $x$.
Given $x=(x_i)_{i\in\I_d}\in\R^d$ we denote by $x\da=(x_i\da)_{i\in\I_d}$ the vector obtained by 
rearranging the entries of $x$ in non-increasing order. We also use the notation
$(\R^d)\da=\{x\in\R^d\ :\ x=x\da \}$ and $(\R_{\geq 0}^d)\da=
\{x\in\R_{\geq 0}^d\ :\ x=x\da \}$. For $r\in\N$, we let $\uno_r=(1,\ldots,1)\in\R^r$.

\pausa
 Given a matrix $A\in\matsad$ we denote by $\la(A)=\la\da(A)=(\la_i(A))_{i\in\I_d}\in (\R^d)\da$ 
the eigenvalues of $A$ counting multiplicities and arranged in 
non-increasing order.   
For $B\in\mat$ we let $s(B)=\la(|B|)$ denote the singular values of $B$, i.e. the eigenvalues of $|B|=(B^*B)^{1/2}\in\matpos$; we also let $\sigma(B)\subset \C$ denote the spectrum of $B$.
If $x,\,y\in\C^d$ we denote by $x\otimes y=x\,y^*\in\mat$ the rank-one matrix given by $(x\otimes y) \, z= \langle z\coma y\rangle \ x$, for $z\in\C^d$.

\pausa Next we recall the notion of majorization between vectors, that will play a central role throughout our work.
\begin{fed}\rm 
Let $x\in\R^k$ and $y\in\R^d$. We say that $x$ is
{\it submajorized} by $y$, and write $x\prec_w y$,  if
$$
\suml_{i=1}^j x^\downarrow _i\leq \suml_{i=1}^j y^\downarrow _i \peso{for every} 1\leq j\leq \min\{k\coma d\}\,.
$$
If $x\prec_w y$ and $\tr x = \sum_{i=1}^kx_i=\sum_{i=1}^d y_i = \tr y$,  then $x$ is
{\it majorized} by $y$, and write $x\prec y$.
\end{fed}

\begin{rem}\label{desimayo}
\pausa Given $x,\,y\in\R^d$ we write
$x \leqp y$ if $x_i \le y_i$ for every $i\in \mathbb I_d \,$.  It is a standard  exercise 
to show that: 
\begin{enumerate}
\item $x\leqp y \implies x^\downarrow\leqp y^\downarrow  \implies x\prec_w y $. 
\item $x\prec y\implies |x|\prec_w|y|$, where $|x|=(|x_i|)_{i\in\I_d}\in\R_{\geq 0}^d$.
\item $x\prec y $ and $ |x|\da=|y|\da \implies x\da=y\da$. 
\item $x\prec y$ and $z\prec w\in \R^e \implies  (x,z)\prec (y,w)\in\R^{d+e}$. 
\EOE\end{enumerate}
\end{rem}

\pausa 
Although majorization is not a total order in $\R^d$, there are several fundamental inequalities in 
matrix theory that can be described in terms of this relation. As an example of this phenomenon we can consider 
Lidskii's (additive) inequality (see \cite{Bhat}). In the following result we also include 
the characterization of the case of equality obtained in \cite{mrs2}.


\begin{teo}[Lidskii's inequality]\label{mrs284}\rm
Let $A,\,B\in \matsad$. 
Then 
\ben
\item $\la(A)-\la(B) \prec \la(A-B)$.
\item $\la(A-B)=\big(\,\la(A)-\la(B)\,\big)^{\downarrow}$ if and only if there exists 
$\{v_i\}_{i\in\I_d}$ an ONB of $\C^d$ such that 
\beq
A=\sum_{i\in\I_d}\la_i (A) \ v_i\otimes v_i \py B=\sum_{i\in\I_d}\la_i (B)\ v_i\otimes v_i\  .
\eeq
Notice that in this case, $A$ and $B$ commute. \qed
\een 
\end{teo}
\pausa
Recall that a norm $N$ in $\mat$ is {\bf unitarily invariant} (briefly u.i.n.) if 
$$ 
\nui{UAV}=\nui{A} \peso{for every} A\in\mat \py U,\,V\in\matud\ ,
$$
and $N$ is {\bf strictly convex} if its restriction to diagonal matrices is a strictly convex  norm in $\C^d$. 
Examples of u.i.n. are the spectral norm $\|\cdot\|$ and the $p$-norms $\|\cdot\|_p$, for $p\geq 1$
(strictly convex if $p>1$).
It is well known that (sub)majorization relations between singular values of matrices are intimately related 
with inequalities with respect to u.i.n's. 
The following result summarizes these relations (see for example \cite{Bhat}):

\begin{teo}\label{teo intro prelims mayo}\rm
Let $A,\,B\in\mat$ be such that $s(A)\prec_w s(B)$. Then:
\ben 
\item For every u.i.n. $N$ in $\mat$
we have that $N(A)\leq N(B)$.
\item If $N$ is a strictly convex u.i.n. in $\mat$ 
and $N(A)=N(B)$, then  $s(A)=s(B)$.\qed
\een

\end{teo}

\subsection{Finite frames}
We consider some notions and results from the theory of finite frames. In what follows we adopt: 

\pausa
{\bf Notation and terminology}: let $\cF=\{f_i\}_{i\in\I_k}$ be a finite sequence in $\C^d$. Then,
\ben
\item $T_\cF\in \cM_{d,k}(\C)$ denotes the synthesis operator of $\cF$ given by $T_\cF\cdot(\alpha_i)_{i\in\I_k}=\sum_{i\in\I_k}\alpha_i\, f_i$.
\item $T_\cF^*\in \cM_{k,d}(\C)$ denotes the analysis operator of $\cF$ and it is given by $T_\cF^*\cdot f=(\langle f,f_i\rangle)_{i\in\I_k}$.
\item  $S_\cF\in \matpos$ denotes the frame operator of $\cF$ and it is given by $S_\cF=T_\cF\,T_\cF^*$. Hence, 
\beq\label{el SF}
S_\cF =\sum_{i\in\I_k} f_i\otimes f_i 
\py  R(S_\cF) =  \gen\{f_i : i \in \I_k\} \,.
\eeq 

\item We say that $\cF$ is a frame for $\C^d$ if it spans $\C^d$; equivalently, $\cF$ is a frame for $\C^d$ if $S_\cF$ is a positive invertible operator acting on $\C^d$.
\EOE\een 

\pausa
Hence, in case $\cF=\{f_i\}_{i\in\I_k}$ is a frame for $\C^d$ we get the so-called canonical reconstruction formulas: for $x\in\C^d$,
$$ 
x=\sum_{i\in\I_k} \langle x\coma  S_\cF^{-1}f_i\rangle \, f_i =\sum_{i\in\I_k} 
\langle x\coma  f_i\rangle \, S_\cF^{-1}f_i\ .$$
In several applications of finite frame theory, it is important to construct families
$\cF=\{f_i\}_{i\in\I_k}\in (\C^d)^k$ in such a way that the frame operator $S_\cF$ and the squared norms $(\|f_i\|^2)_{i\in\I_k}$
are prescribed in advance. This problem is known as the frame design problem, and its solution can be obtained in 
terms of the Schur-Horn theorem for majorization.

\begin{teo}[See \cite{Illi}]\label{teo SH para marcos}
Let $S\in \matpos$ and let $\ca=(a_i)_{i\in\I_k}\in (\R_{>0}^k)$. Then, the following statements are equivalent:
\ben
\item There exists $\cF=\{f_i\}_{i\in\I_k}\in (\C^d)^k$ such 
that $S_\cF=S$ and $\|f_i\|^2=a_i\,$, for $i\in\I_k\,$;
\item $\ca\prec \la(S)$.	
\qed\een
\end{teo}

\section{Generalized frame operator distance functions}\label{sec tris}

In this section we state our main problem namely, the study of the geometrical and spectral structure
of local minimizers of generalized frame operator distance (GFOD) functions. After recalling some preliminary results
from \cite{dnp2}, we obtain a description of what we call the inner structure of local minimizers of GFOD's functions.
Since the proofs of some results in this section 
are quite technical, they are developed in Section \ref{Appendixity}.

\subsection{Statement of the problem and related results}

Let $S\in\matpos$ and $\asubi$. In this case we consider the torus 
$$
\tcal \igdef \llav{ \cG=\{g_i\}_{i\in\I_k} \in  (\C^d)^{k} \  : 
\norm{g_{i}}^2=a_i  \,\mbox{, for every }   i \in \IN{k}}\  . 
$$ 
By definition, $\tcal$ is the (cartesian) product of spheres in $\C^d$; 
we endow $\tcal$ with the product metric of the Euclidean metrics in each of these spheres, namely 
$$
d^2(\cG\coma \cG')=\sum_{i\in\I_k}\|g_i-g_i'\|^2  \peso{for} 
\cG=\{g_i\}_{i\in\I_k},\, \cG'=\{g_i'\}_{i\in\I_k}\in\tcal\ .
$$  
Thus, $\tcal$ is a compact smooth manifold.
Given a strictly convex u.i.n $N:\mat\rightarrow  \R_{\geq 0}$, we can consider 
the generalized frame operator distance (G-FOD) in $\tcal$ (see \cite{dnp}) given by
$$
\Theta_{(N\coma S \coma \ca)}=\Theta: \tcal \rightarrow \R_{\geq0} \peso{given by} \Theta(\cG)=N(S-S_{\cG})\ ,
$$
where $S_\cG=\sum_{i\in\I_k} g_i\otimes g_i$ denotes the frame operator of a family $\cG\in\tcal$. 
This notion is based on the frame operator distance (FOD) $\Theta_{(\|\cdot\|_2\coma S\coma \ca )}$ 
introduced by Strawn in \cite{strawn}, where
$\|A\|_2^2=\tr(A^*A)$ denotes the Frobenius norm, $A\in\mat$. 
Based on his work and on numerical evidence, Strawn conjectured that local minimizers of 
$\Theta_{(\|\cdot\|_2\coma S\coma \ca )}$ 
were also global minimizers. In \cite{dnp} we settled Strawn's conjecture in the affirmative, by relating 
FOD problems in the norm $\|\cdot\|_2$ with optimal frame completion problems for the Benedetto-Fickus frame potential.
It is then natural to ask whether local minimizers of the G-FOD $\Theta_{(N\coma S\coma \ca )}$ are also global minimizers, where
$N$ denotes an arbitrary strictly convex u.i.n. on $\mat$ (e.g. $p$-norms, with $p\in(1,\infty)$). Unfortunately, the techniques used in \cite{dnp} do not apply in this general case, leaving untouched the 
following 
\begin{probs}
Let $S\in\matpos$, $\asubi$ and fix a strictly convex u.i.n. $N$ on $\mat$. Then
\begin{itemize}
\item[P1.] Compute the spectral and geometrical structure of local minimizers of $\Theta_{(N\coma S\coma \ca )}$ 
in $\tcal$.
\item[P2.] Determine whether local minimizers are global minimizers of $\Theta_{(N\coma S\coma \ca )}$ 
in $\tcal$.
\item[P3.] Determine whether these minimizers depend on the chosen u.i.n.
\qed\end{itemize}
\end{probs}

\pausa 
In what follows we completely solve the three problems above in an algorithmic way, thus settling in the affirmative 
the questions in P2. and P3. (see Theorem \ref{si kmenord} in Section \ref{sec kmenord}).

\pausa
Next, we recall some results from 
\cite{dnp2} that we use throughout our work.

\begin{teo}[See \cite{dnp2}]\label{teo applic1} \rm
Fix $S\in\matpos$,  $\asubi$, 
and a strictly convex u.i.n. $N$ on $\mat$. 
Consider the map 
$\tnsa=\Theta: \tcal \rightarrow \R_{\geq0}$ given by $\Theta(\cG)=N(S-S_{\cG})$.

\pausa
Fix a local minimizer $\cG_0=\llav{g_i}_{i\in\I_k}\in\tcal$ of $\tnsa$, with frame operator 
$S_0=S_{\cG_0}\,$. Denote by  $W=R(S_0)=\text{span}\{g_i:\ i\in\I_k\} \inc \C^d$. 
Then,
\begin{enumerate}
\item There exists $\mathcal B=\{v_i\}_{i\in\I_d}$ an ONB of $\C^d$ such that 
$$
S=\sum_{i\in\I_d} \la_i(S) \ v_i\otimes v_i
\py  S_0=\sum_{i\in\I_d} \la_i(S_0) \ v_i\otimes v_i  
\,.
$$ 
In particular, we have that $\la(S-S_0)=\big(\la(S)-\la(S_0)\,\big)\da$.
\item The subspace $W$ reduces $S-S_0\in\cH(d)$; hence, $D\igdef (S-S_0)|_W\in L(W)$ verifies $D^*=D$.
\item All vectors $g_i$ ($i\in\I_k$) are eigenvectors of $D$ and $S-S_0\,$. 
\item Let $\sigma(D)=\{c_1\coma \ldots\coma c_p\}$ 
be such that $c_1<c_2<\ldots<c_p\,$. Denote by 
$$
J_j=\{\ell \in \I_k: \ D\,g_\ell=c_j\, g_\ell\} \py 
W_j=\text{span}\{g_\ell:\ \ell\in J_j\} \peso{for} j\in\I_p\ .$$
Then the subspaces $W_j$ reduces both 
$S$ and $S_0$, for $j\in\I_p\,$. Moreover, 
\beq\label{Wj ort}
\I_k = 
\stackrel{\text{D}}{\bigcup_{j\in\I_p}} \ J_j  \peso{(disjoint union) \ \ and} W=\bigoplus_{j\in\I_p} W_j
\ .
\eeq
\item If $j\in \I_p$ and $c_j \neq \max \sigma(S-S_0)$ (for example, when $1\leq j<p$), then 
the family $\{g_\ell\}_{\ell\in J_j}$ is linearly independent. 
\qed\end{enumerate}
\end{teo}

\begin{rem}\label{cp max}
 With the notation of Theorem \ref{teo applic1}, if we assume that 
\beq\label{cp max eq}
k\ge d  \quad \implies \quad  
c_p =\max \sigma(S-S_0)  \ . 
\eeq
Indeed, if $W = \C^d$ then 
$\sigma(S-S_0) = \{c_1\coma \ldots\coma c_p\}$. Otherwise $\dim W<d\le k$ so, by items 4 and  5 of Theorem \ref{teo applic1}, 
the family $\{g_i\}_{i\in J_p}$ can not be linearly independent (because the families 
$\{g_i\}_{i\in J_j}$ are linearly independent for $1\leq j<p$, and all families are mutually orthogonal). By item 5 again,   
we deduce  that $c_p=\max \sigma(S-S_0)$. \EOE
\end{rem}

\subsection{Inner structure of local minimizers of GFOD's}\label{sec inner struc}

In this section, based on Theorem \ref{teo applic1} above, we obtain a detailed description of 
what we call the inner structure of local minimizers. In order to do this, we introduce the following
\begin{nota} \label{nuevas notas 2} 
Fix $S\in\matpos$,  $\asubi$ 
and a strictly convex u.i.n. $N$ on $\mat$. Also consider 
the notions introduced in Theorem \ref{teo applic1}. As before, consider
\begin{enumerate}
\item $\tnsa=\Theta: \tcal \rightarrow \R_{\geq0}$ given by $\Theta(\cG)=N(S-S_{\cG})$.
\item A local minimizer $\cG_0=\llav{g_i}_{i\in\I_k}\in\tcal$ of $\tnsa$, with frame operator 
$S_0=S_{\cG_0}$.
\item We denote by  $\la=(\la_i)_{i\in\I_d} = \la(S) \in (\R_{\geq 0}^d)\da$  and 
$\mu=(\mu_i)_{i\in\I_d} = \la(S_0) \in (\R_{\geq 0}^d)\da$.

\item We fix $\cB=\{v_i\}_{i\in\I_d}$ an ONB of $\C^d$ as in Theorem \ref{teo applic1}. Hence, 
\beq\label{la y mu}
S=\sum_{i\in\I_d} \la_i \ v_i\otimes v_i
\py S_{0}=\sum_{i\in\I_d} \mu_i \ v_i\otimes v_i  \ .
\eeq 
\item We consider
$W=R(S_0)$, $D=(S-S_0)|_W$ and 
$\sigma(D)=\{c_1\coma\ldots\coma c_p\}$ where $c_1<c_2<\ldots<c_p$. 

\item Let $ s_D= \max \, \{i\in\I_d: \mu_i \neq 0\} = \rk \, S_{0}\,$.

\item We denote by $  \delta= \la - \mu \in \R^d$ 
so that, by Eq. \eqref{la y mu},  
$$
S-S_0=\sum_{i\in\I_d} \delta_i \ v_i\otimes v_i \peso{and}
D=\sum_{i=1}^{s_D} \delta_i \ v_i\otimes v_i\ .
$$
Notice that $\delta$ is constructed by pairing the entries 
of ordered vectors (since $\la=\la(S)$ and $\mu=\la(S_0)$. 
Nevertheless, we have that $\la(S-S_0) = \delta\da$. 
In what follows we obtain some properties of (the unordered vector) $\delta$. 

\item For each $j\in \I_p\,$, we consider the following sets of indexes: 
$$K_j = \{ i \in \I_{s_D} :   \delta_i=\la_i -\mu_i = c_j\}  
\py J_j = \{i\in \I_k: D\,g_i = c_j \, g_i\}  \ .
$$ 
Theorem \ref{teo applic1} assures that   
$\I_{s_D}  
= \stackrel{\text{D \ \ \ \ \ }}{\bigcup_{j\in\I_p}} \ K_j \py \I_k  
= \stackrel{\text{D \ \ \ \ \ }}{\bigcup_{j\in\I_p}} \ J_j$ (disjoint unions).
\item By Eq. \eqref{el SF}, $R(S_{0})= \gen\{g_i : i \in \I_k\} 
= W = \bigoplus_{i\in\I_p} \ker \,(D-c_i\,I_W\,)$. 
Then, for every $ j\in \I_p\ $, 
$$
W_j =\gen\{g_i : i \in J_j\} = \ker \,(D-c_j\,I_W\,) = \gen\{v_i : i \in K_j\} \ ,
$$
because $g_i \in  \ker \,(D-c_j\,I_W\,)$ for every $i \in J_j\,$. 
Note that, by Theorem \ref{teo applic1}, each $W_j$ reduces both $S$ and $S_{0}\,$. \EOE 
\end{enumerate}
\end{nota}

\pausa The next proposition describes the structure of the sets $J_j$ and $K_j$ for $j\in\I_p\,$, as defined in 
Notation \ref{nuevas notas 2}. In turn, these sets play a central role in the proof of Theorem \ref{teo estruc min loc detallada} below.

\begin{pro} \label{nueva pro 3.10 y 3.11}
Let $S\in\matpos$ and $\cG_0\in \tcal$ be as in Notation \ref{nuevas notas 2}. 
Then there exist indexes $0=s_0<s_1<\ldots<s_{p-1}<s_p= \rk \,S_0 \leq d$ such that 
\beq\label{los Kj}
\barr{rl}
K_j &=J_j =\{s_{j-1}+1\coma \ldots \coma  s_j\} \ ,\quad 
\peso{for} j\in\IN{p-1}\ \  (\text{if} \ p>1 ) , \\&\\
K_p& =\{s_{p-1}+1 \coma \ldots \coma  s_p\} \ \ , \ \ J_p=\{s_{p-1}+1\coma \ldots\coma  k\} \ .
\earr 
\eeq
\end{pro}
\begin{proof}
See Section \ref{Appendixity}.
\end{proof}

\pausa
\begin{rem}\label{el delta}\rm
Consider Notation \ref{nuevas notas 2} for $S\in\matpos$ and a local minimizer $\cG_0\in \tcal$ of the map $\tnsa\,$. 
Let 
$s_0=0<s_1<\ldots<s_p\leq d$, where $s_p=\rk(S_0)$, be as in Proposition \ref{nueva pro 3.10 y 3.11}. 
In terms of these indexes we also get that $\la(S-S_0)= \delta (S\coma \ca \coma \G_0)\da$ for 
$\delta (S\coma \ca \coma \G_0) = \la(S)-\la (S_0)$, and 

\beq\label{eq el la de F}
\delta (S\coma \ca \coma \G_0) =
\big(\, c_1 \, \uno_{s_1} \coma c_2 \, \uno_{s_2-s_{1}} \coma 
\dots \coma c_p\, \uno_{s_p-s_{p-1}} \coma \la_{s_p+1} \coma 
\dots \coma \la_d\,\big) \peso{if} s_p<d\,
\eeq or 
\beq\label{eq el la de F2}
\delta (S\coma \ca \coma \G_0) = \big(\, c_1 \, \uno_{s_1} \coma c_2 \, \uno_{s_2-s_{1}} \coma 
\dots \coma c_p\, \uno_{s_p-s_{p-1}} \,\big)\peso{if} s_p=d\,.
\eeq 

\pausa
In the next result, we obtain a characterization of the indexes $s_1<\ldots<s_{p-2}$ and constants $c_1<\ldots c_{p-1}$ 
in terms of the index $s_{p-1}\,$ (when $p>1$). 
In the next section we complement these results and show the key role played by the index $s_{p-1}$ and give 
a characterization of $c_p\,$. We begin by fixing some  notation, which is independent of the norm $N$ and the 
local minimizer $\cG_0$.
\EOE
\end{rem}
\begin{nota}\label{nota los his} \rm
Let $S\in\matpos$,  $\asinsubi$, $\la(S)=(\la_i)_{i\in\Id}\in (\R^d)\da$ 
and $m =\min\{k\coma d\}$.
\ben
\item We let $h_i \igdef \la_{i}-a_i\,$, for every $i\in\I_m\,$. 
\item Given  $1\leq j\leq r \leq m$, let 
\beq\label{Pjr}
P_{j\coma r} =\frac{1}{r-j+1}\ \sum_{i=j}^r\  h_i = 
\frac{1}{r-j+1}\ \sum_{i=j}^r\ \la_{i}-a_i 
\ .
\eeq
 We abbreviate $P_{1 \coma r} = P_r\,$ for the initial averages. \EOE  
\een\end{nota}

\begin{teo}\label{teo estruc min loc detallada} 
Consider Notation \ref{nuevas notas 2} for $S\in\matpos$ and a local minimizer $\cG_0\in \tcal$ of the map $\tnsa\,$.
Assume further that $p>1$. Let 
$s_0=0<s_1<\ldots<s_p\leq d$ be such that Eq. \eqref{los Kj} holds.
Then, we have the following relations:
\ben
\item The index $s_1 = \max \, \big\{1\leq r \le s_{p-1} \, :\, 
P_{r} = \min\limits_{1\leq i\le s_{p-1}}  \, P_{ i} \, \big\}$, and 
$c_1 = P_{ s_1}\,$.
\item Recursively, if $s_j<s_{p-1}\,$, then
$$
s_{j+1} = \max \, \big\{s_j< r \le s_{p-1} \, :\, 
P_{s_j+1\coma r} = \min\limits_{s_j< i\le s_{p-1}}  \, P_{s_j+1 \coma i} \, \big\} 
\py c_{j+1} = P_{s_j+1\coma s_{j+1}}\ .
$$
\een
\end{teo}
\begin{proof}
See Section \ref{Appendixity}.
\end{proof}
\subsection{The co-feasible case for $k\ge d$.}\label{sec cofea}

Throughout this section we assume that $k \ge d$.  
In \cite{dnp2} we showed that in some cases, local minimizers of G-FOD functions are also global minimizers. We recall this fact
in the following

\begin{teo}[See \cite{dnp2}]\label{teo conj caso esp}\rm
Consider Notation \ref{nuevas notas 2} with $k\ge d$ for $S\in\matpos$ and a local minimizer $\cG_0\in \tcal$ 
of the map $\tnsa\,$. 
 Assume further that $p=1$ i.e., that 
there exists $c=c_1$ that satisfies $(S-S_{0}) g_i=c\,g_i\, $, for every $i\in\I_k\,$. 
Then there exists an ONB $\{v_i\}_{i\in\I_d}$ of $\C^d$ such that 
$$S=\sum_{i\in\I_d} \la_i\ v_i\otimes v_i \py S_{0}=\sum_{i\in\I_d} (\la_i-c)^+\ v_i\otimes v_i\,, $$
where $(\la_i)_{i\in\I_d}=\la(S)\in (\R_{\geq 0}^d)\da$. 
Moreover,
$\cG_0$ is a global minimizer of $\Theta$ in $\tcal$.
\qed
\end{teo}

\begin{cor}\label{cor cons caso cofea} 
With the hypotheses and notation in Theorem \ref{teo conj caso esp} we have that:
\begin{enumerate}
\item The constant $c=\max \sigma(S-S_0)$ is the largest eigenvalue of $S-S_0$.
\item The eigenvalue $\la_i(S_0)= (\la_i-c)^+$, for every $i\in\I_d\,$.
\item The list of norms  $\ca\prec \big(\, (\la_i-c)^+\big)_{i\in\I_d}\,$.  In particular 
$$\tr(\ca)=\sum_{i\in\I_k}a_i=\sum_{i\in\I_d} (\la_i-c)^+\,.$$
\end{enumerate}
\end{cor}
\begin{proof}
1. 
We are assuming that  $k\geq d$. Then  Remark  \ref{cp max} 
assures that $c=c_p=\max \sigma(S-S_0)$.

\pausa
2. This is a direct consequence of Theorem \ref{teo conj caso esp} above and the fact that $(\la_i)_{i\in\I_d}=\la(S)\in (\R^d)\da$, 
so that also $\big(\, (\la_i-c)^+\big)_{i\in\I_d}\in (\R^d)\da$.

\pausa
3. Since $\cG_0 \in \tcal$ (it is a family of vectors with norms given by $\ca$), 
then  Theorem \ref{teo SH para marcos} assures that 
$$
\ca\prec \la(S_{\cG_0}) = \big(\, (\la_i-c)^+\big)_{i\in\I_d}\ .
$$ 
The rest of the statement
is a direct consequence of this majorization relation.
\end{proof}

\pausa
The previous results motivate the following notion, which only depends on some $\la  \in (\R_{\ge0}^{d})\da$
and $\asubi$, with $k\ge d$ (and does not require any norm $N$ nor a local minimizer $\cG_0$).

\begin{fed}\label{caso cofea}\rm
Let $\la  \in (\R_{\ge0}^{d})\da $  
and $\asubi$, with $k\ge d$. 
We say that the pair $(\la\coma \ca) $ is {\bf co-feasible} if 
there exists a constant 
\beq\label{c cof}
c< \la_1  \peso{such that} \ca\prec \big(\, (\la_i-c)^+\big)_{i\in\I_d} \ .
\eeq
In this case, 
the co-feasibility constant $c$ is uniquely determined by 
$
\tr(\ca)=\sum\limits_{i\in\I_d} (\la_i-c)^+\,
$.
\EOE
\end{fed}

\begin{pro}\label{prop cond suf para cofea}
Let $S\in\matpos$ and $\asubi$ with $k\ge d$. Then the pair $(\la(S)\coma \ca) $ is co-feasible if and only if
the following conditions hold:
\ben
\item There exist $\cG=\{g_i\}_{i\in\I_k}\in\tcal$ and $c\in \R$ such that $(S-S_\cG)\, g_i=c\,g_i\,$, for every  $i\in\I_k\,$.
\item This constant $c=\max \sigma(S-S_\cG)$.\EOE
\een
\end{pro}
\begin{proof}
Assume that there exist 
$c\in \R$ and  $\cG=\{g_i\}_{i\in\I_k}\in\tcal$ 
which satisfy items 1 and 2. 
By Eq. \eqref{el SF},   $W = R(S_\cG)= \gen\{g_i : i\in \I_k\}$. Since 
 $(S-S_\cG)\big|_W =c\, I_W\,$, then $S(W)\inc W$. Let $r = \dim W$. 
Then, considering separately the eigenvalues of $S\big|_W $
and $S\big|_{W\orto}=(S-S_\cG)\big|_{W\orto}\,$,   the fact that 
$c=\max \sigma(S-S_\cG)$ implies that  
$$
c< \la_i(S)  \peso{for} i \in \I_r \py c\ge 
\la_i(S) \peso{for} r<i\le d  \ .
$$
Therefore 
$\la(S_\cG) = \la (S- (S-S_\cG)\,) = \big(\, (\la_i(S)-c)^+\big)_{i\in\I_d}$. 
Hence,  arguing as in the proof of Corollary 
\ref{cor cons caso cofea}, we conclude that this $c$ satisfies Eq. \eqref{c cof}. Note that  $c<\la_1(S)$ because $\tr \ca \neq 0$. 

\pausa
Conversely, if there exists $c$ which satisfies Eq. \eqref{c cof}, let $\cB=\{v_i\}_{i\in\I_d}$ be an ONB for $\C^d$ such that 
$$
S=\sum_{i\in\I_d} \la_i(S)\, v_i\otimes v_i \peso{,  \ \ and set }S_0\igdef \sum_{i\in\I_d} (\la_i(S)-c)^+\, v_i\otimes v_i\in\matpos\ .
$$
By Theorem \ref{teo SH para marcos}, there exists $\cG=\{g_j\}_{j\in\I_k}\in\tcal$ such that $S_0=S_\cG\,$.
Note that 
\beq\label{minmin}
\la_i(S)-(\la_i(S)-c)^+=\min\{\la_i(S)\coma c\}  \peso{for every}  i\in\I_d\ . 
\eeq
Then $c\ge\max \sigma(S-S_\cG)$. 
If we let 
$$
r=\rk \, S_\cG=\max\{i\in\I_d:\ (\la_i(S)-c)^+>0\}=\max\{i\in\I_d:\ \la_i(S)> c\}\ge 1 \ ,
$$ 
then $\trivial \neq W\igdef  R(S_\cG)=\text{span}\{v_i:\ i\in\I_r\}$, and it  
satisfies that $(S-S_\cG)\big|_W = c\, I_W\,$. 
The proof finishes by noticing that, by Eq. \eqref{el SF}, 
 $g_j \in W $ and hence $(S-S_\cG)g_j=c\, g_j$ for every $j\in \I_k\,$.
\end{proof}

\begin{rem} \rm 
Let $S\in\matpos$ and $\asubi$ (with $k\ge d$) such that the pair $(\la(S)\coma \ca) $ is co-feasible. 
Let $\cG=\{g_i\}_{i\in\I_k}\in\tcal$ and $c\in \R$ be as in the proof of the 
second part of Proposition \ref{prop cond suf para cofea}. 
Then,  by Theorem \ref{teo conj caso esp}, $\cG$ is a global (and local) 
minimizer of the map $\tnsa\,$, with $p =1$.  Nevertheless, a priori this fact does not imply that 
every local minimizers should have the same structure (namely, to have also $p=1$). We shall prove soon 
that the spectral structure of local minimizers is indeed unique (in general, and then also in the co-feasible cases). 
 \EOE
\end{rem}

\pausa
It is worth pointing out that there are GFOD problems that are not co-feasible. In order to see this
we include the following:

\begin{exa}\label{exa no cofeasible}
Consider $S\in\cM_4(\C)^+$ be such that $\la:=\la(S)=(2,2,1,1)\in(\R_{>0}^4)\da$ and let $\ca=(3,1,1,1)\in(\R_{>0}^4)\da$. Then, the pair
$(\la\coma\ca)$ is not co-feasible. Indeed, the unique solution $c<2$ to the equation $6=\tr(\ca) = 2\,(2-c)^++ 2\,(1-c)^+$ is $c=0$.
Thus $((\la_i-c)^+)_{i\in\I_4}=\la$. But it can be easily checked that $\ca\not\prec \la $. 
\EOE 
\end{exa}

\pausa
Although in general, given $S\in\matpos$ and $\asubi$, the pair $(\la(S)\coma \ca)$ corresponding to this data is 
not co-feasible, 
the GFOD problems contain a co-feasible part. Indeed, if we further consider a strictly convex u.i.n. $N$ in $\mat$, 
then local minimizers of $\Theta_{(N\coma S\coma \ca)}$ allow us to locate such co-feasible parts. In order to describe
this situation, we introduce the following

\begin{fed}\label{cofea index}\rm
Let $S\in\matpos$ and $\asubi$ with $k\ge d$. 
For $r\in\I_{d-1}\cup \trivial$ we consider the truncated data 
$$
\la^{(r)}(S) =(\la_{r+1}(S) \coma \ldots \coma \la_d(S))
\in(\R_{\geq 0}^{d-r})\da \py 
\ca^{(r)} =(a_{r+1}(S) \coma \ldots \coma a_k) \in (\R_{>0}^{k-r})\da \ .
$$
We say that $r$ is a {\bf co-feasible index} for $S$ and $\ca$ if the pair 
$( \la^{(r)}(S) \coma \ca^{(r)}) $ is co-feasible
 (according to Definition \ref{caso cofea} with dimensions $d-r\leq k-r$).
\EOE
\end{fed}

\begin{rem} \label{R3.15}\rm
Let $S\in\matpos$ and $\asubi$ with $k\ge d$. 
Let $\cB=\{v_i\}_{i\in\I_d}$ be an ONB for $\C^d$ such that $S\,v_i=\la_i(S)\, v_i$ for $i\in\I_d\,$.
Then, by Proposition \ref{prop cond suf para cofea}, an index   $r\in\I_{d-1}\cup \trivial$ is co-feasible \sii
the conditions $1$ and $2$ of Proposition \ref{prop cond suf para cofea} hold for the space 
$V_r=\text{span}\{v_i:\ r+1\leq i\leq d\}$, the positive operator and $S_r=S|_{V_r}\in L(V_r)$
and the vector of norms  $\ca^{(r)} =(a_{r+1}(S) \coma \ldots \coma a_k) \in (\R_{>0}^{k-r})\da $. 
This means that there exist  $c\in \R$ and 
$$
\cG=\{g_i\}_{i\in\I_{k-r}}\in\mathbb T_{V_r}(\ca^{(r)}) \igdef \mathbb T_{k-r}(\ca^{(r)}) \cap  V_r^{k-r} \py c\in \R 
$$
such that $(S_r-S_\cG)\, g_i=c\,g_i\,$, for every  $i\in\I_{k-r}\,$, and  $c=\max \sigma(S_r-S_\cG)$.
Note that this statement seems to depend on the basis $\cB$. 
But actually, the list of eigenvalues  $\la(S_r)= \la^{(r)}(S) 
\in(\R_{\geq 0}^{d-r})\da$, so it does not depend on $\cB$.
\EOE
\end{rem}
%
%
%
%
%
%

\pausa 
The next result complements Theorem \ref{teo estruc min loc detallada}. 

\begin{pro} \label{pro index fea para loc min}
Consider Notation \ref{nuevas notas 2} with $k\ge d$ for $S\in\matpos$ and a local minimizer $\cG_0\in \tcal$ of the map $\tnsa\,$. 
Let $0=s_0<s_1<\ldots<s_{p-1}<s_p\leq d$ be as in Proposition \ref{nueva pro 3.10 y 3.11}.
Then $c_p=\max\sigma(S-S_{\cG_0})$ and $s_{p-1}$ is a co-feasible index for $S$ and $\ca$. 

\pausa
In particular, the constant $c_p$ and the index $s_p= \rk\, S_{\cG_0}$ are uniquely determined by 
the equations 
\beq\label{eq para cp}
\sum_{i=s_{p-1}+1}^k a_i=\sum_{i=s_{p-1}+1}^d (\la_i(S)-c_p)^+ \py
s_p=\max\{s_{p-1}+1\leq i\leq d\ : \ \la_i(S)-c_p>0\}\,.
\eeq
\end{pro}
\begin{proof} Let $S_0 = S_{\cG_0}\,$. 
Note  that $c_p=\max\sigma(S-S_0)$ by Remark \ref{cp max}, since we are assuming that $k\ge d$. 
In order to show that $s_{p-1}$ is a co-feasible index we shall use Remark \ref{R3.15}. Let $r=s_{p-1}\,$. Recall from 
Notation \ref{nuevas notas 2} and Proposition \ref{nueva pro 3.10 y 3.11} that 
$J_p = \{i \in \I_k : (S-S_0) g_i = c_p\, g_i\} = \{r+1\coma \ldots\coma  k\}$ and that 
$W_p = \gen\{g_i: i \in J_p\} =\gen \{v_j : r+1\le j \le s_p\}$. 
Since  
$$
W = R(S_0)= \gen \{v_i :  j \in \I_{s_p}\} \peso{then} 
V_r=\text{span}\{v_i:\ r+1\leq i\leq d\} = 
W_{p}\oplus W^\perp \ . 
$$ 
Then, $\cG_r=\{g_i\}_{i=r+1}^k\in\mathbb T_{V_r}(\ca^{(r)})= \mathbb T_{k-r}(\ca^{(r)}) \cap  V_r^{k-r}$ 
is such that $S_{\cG_r} = 
S_0|_{V_r}$ (here we use that, by Eq. \eqref{Wj ort},  
$g_j \in W_p\orto $ for every $j\notin J_ p$). So that, if $P_\cM$ denotes the orthogonal projection onto a subspace $\cM\inc \C^n$,  
$$
S|_{V_r}-S_{\cG_r}=(S-S_0)|_{V_r}=c_p\,P_{W_p}+ S\,P_{W^\perp} 
\implies (S|_{V_r}-S_{\cG_r})\,g_i=c_p\, g_i \peso{for} r+1\leq i\leq k\ .
$$
Hence $\max\sigma(S|_{V_r}-S_{\cG'})\leq \max\sigma(S-S_0)=c_p\,$ and, by Remark \ref{R3.15}, 
$s_{p-1}=r $ is a co-feasible 
index for $S$ and $\ca$. Then, by Definition  \ref{cofea index}, 
$s_p$ and $c_p$ are determined by Eq. \eqref{eq para cp}.
\end{proof}

\begin{rem} \label{el delta con N}\rm
Consider Notation \ref{nuevas notas 2} with $k\ge d$ for $S\in\matpos$ and a local minimizer $\cG_0\in \tcal$ of the map $\tnsa\,$. 
Taking into account all objects and  facts detailed in Notation \ref{nuevas notas 2}, Remark \ref{el delta}, Theorem \ref{teo estruc min loc detallada}, Eq. \eqref{minmin} and 
 Proposition \ref{pro index fea para loc min}, we conclude that 
$\la (S-S_{\cG_0}) = \delta (S\coma \ca \coma \G_0)\da$, with
\beq\label{del SaN}
\delta(S\coma\ca\coma \cG_0) \igdef
\left( 
c_1\,\uno_{s_1(r)} \coma c_2\,\uno_{s_2-s_1} 
\coma \ldots \coma c_{p-1}\, \uno_{s_{p-1}-s_{p-2}} \coma  
\big(\min\{\la_i(S)\coma c_p \}\big)_{i=s_{p-1}+1}^d
\right)\ , 
\eeq
or $\delta(S\coma\ca\coma \cG_0)  = 
\big(\min\{\la_i(S)\coma c_1 \}\big)_{i\in \I_d}$ (if $p=1$, the co-feasible case), 
where all data in this formula 
can be explicitly computed in terms of $S$, $\ca$ and the index $s_{p-1}\,$. 
Indeed, this expression depends on $\cG_0$ and 
$N$ only through the index $s_{p-1}$ which determines the previous 
indexes and constants by Theorem \ref{teo estruc min loc detallada}, and the co-feasible part which begins at 
$s_{p-1}$, so it determines $s_p$ and  $c_p\,$, by  Proposition \ref{pro index fea para loc min} via Eq. \eqref{eq para cp}.
Hence we shall denote $s_{p-1}= s_{p-1}(\cG_0)$ . 
%
\EOE

\end{rem}

\pausa
We end this section with the following result, which compares  the co-feasibility constants corresponding to different 
co-feasible indexes.

\begin{cor}\label{cor sobre constantes cofeas}
Let $S\in\matpos$ and $\asubi$ with $k\ge d$ and assume that $r,\, s\in\I_{d-1}$ are co-feasible indexes for $S$ and $\ca$. Denote by 
$c(s)$ and $c(r)$ their co-feasibility constants.
Then,   $$s<r  \implies c(s)\geq c(r) \ .$$
\end{cor}
\begin{proof}
By Proposition \ref{prop cond suf para cofea}, 
$\ca^{(s)}\prec \big(\,(\la_i(S)-c(s)\,)^+\big)_{i=s+1}^d$
and $\ca^{(r)}	\prec \big(\,(\la_i(S)-c(r)\,)^+\big)_{i=r+1}^d\,$.
Then 
\beq\label{r y s}
\sum_{i=r+1}^ka_i=\sum_{i=r+1}^d (\la_i(S)-c(r)\,)^+ \py \sum_{i=s+1}^r a_i\leq \sum_{i=s+1}^r(\la_i(S)-c(s)\,)^+ \ .
\eeq
Therefore, 
$$ 
\sum_{i=s+1}^d(\la_i(S)-c(s)\,)^+ = \sum_{i=s+1}^k a_i
\le \sum_{i=s+1}^r(\la_i(S)-c(s)\,)^+ +\sum_{i=r+1}^d (\la_i(S)-c(r)\,)^+
\ .
$$
But if $c(s)< c(r) $ then  $(\la_i(S)-c(s)\,)^+ \ge (\la_i(S)-c(r)\,)^+$ for every $i \in \I_d\,$, 
and moreover, we have that 
 $(\la_{r+1}(S)-c(s)\,)^+ > (\la_{r+1}(S)-c(r)\,)^+$ because 
$\sum\limits_{i=r+1}^ka_i>0 \stackrel{\eqref{r y s}}{\implies}  c(r) <\la_{r+1}(S)$. 
\end{proof}

\section{Main results}\label{sec cuatris}

In this section we state and prove our main result namely, that local minimizers of GFOD's are actually global minimizers. 
This is achieved by considering in detail the results obtained in Section \ref{sec tris} related with the spectral structure of local minimizers
of GFOD's functions, and the notion of co-feasible index. We first consider the case when $k\geq d$. 

\subsection{When $k\geq d$}\label{sec 4.1.}
Throughout this subsection we assume that $k \ge d$. 
Notice that Eqs. \eqref{eq el la de F} and \eqref{eq el la de F2} together with 
Theorem \ref{teo estruc min loc detallada} and Proposition \ref{pro index fea para loc min} give a detailed description of the spectral structure of 
local minimizers of GFOD problems. With the notation of these results, it is worth pointing out the key role played by the (co-feasible) index $s_{p-1}$
in the determination of the complete spectral structure of $S-S_0$ and $S_0$ (see Definition \ref{caso cofea}). 

\pausa
The basic idea for what follows is to replace $s_{p-1}$ by an arbitrary co-feasible index $r$, to reproduce the algorithm given in 
Theorem \ref{teo estruc min loc detallada} and get indexes and constants in terms of $r$ (which a priori  
are not associated to any minimizer $\cG_0$). Then, we  shall
show that there exists a unique ``correct" 
index $r$ (i.e. co-feasible and admissible, see Definition 
\ref{fed el espectro posta abstracto} below) 
which only depends on $\la(S)$ and $\ca$, so that  it 
must coincide with $s_{p-1}(\cG_0)\,$.

\begin{fed}\label{fed el espectro posta abstracto}\rm
Let $S\in\matpos$ and $\ca=(a_i)_{i\in\I_k}\in(\R_{>0}^k)\da$. For a co-feasible index $r\in\I_{d-1}\cup\{0\}$ 
let $q=q(r)\in\I_d$, 
$0 =s_0(r)<s_1(r)<\ldots<s_{q-1}(r)=r<s_q\leq d\le k$ and $c_1(r),\ldots,\,\,c_q(r)$ 
be computed according to the following recursive algorithm (which only depends on $r$, $\la(S)$ and $\ca$):
\ben
\item If $r=0$, set $q=q(r)= 1$ and $s_0(r) = s_{q-1}(r)=r=0 $ (and go to item 4.). 
\item If $r>0$, using the numbers $P_{i\coma j}$ defined in Notation \ref{nota los his}, 
the index 
$$s_1(r) = \max \, \big\{1\leq j \le r \, :\, 
P_{1\coma j} = \min\limits_{i\le r}  \, P_{1\coma i} \, \big\} \ \ , \py 
c_1(r) = P_{1\coma s_1(r)}\ .
$$
\item If the index $s_j(r)$ is already computed and $s_j(r)<r\,$, then
\beq\label{sjr}
s_{j+1}(r) = \max \, \big\{s_j(r)< j \le r \, :\, 
P_{s_j(r)+1\coma j} = \min\limits_{s_j(r)< i\le r}  \, P_{s_j(r)+1 \coma i} \, \big\} \ ,
\eeq
and $ c_{j+1}(r) = P_{s_j(r)+1\coma s_{j+1}(r)}\,$.
\item If 
$s_j(r)=r\,$, we set $q= q(r) = j+1$ (so that $s_{q-1}(r)=r$), 
and we define 
$c_q(r)$ and $s_q(r)$ (with 
$c_q(r)<\la_{r+1} 
$ and $r= s_{q-1}(r)<s_q(r)\leq d$) that are uniquely determined by
\beq\label{eq para cq1}
\sum_{i= r+1}^k a_i=\sum_{i=r+1}^d \big(\la_i(S)-c_q(r)\,\big)^+ \py
\eeq 
\beq\label{eq para cq2}
s_q(r)=\max\{r+1\leq i\leq d: \la_i(S)-c_q(r)>0\}\ .
\eeq 
In particular, $s_q(r)=\max\{i\in\I_d: \la_i(S)-c_q(r)>0\}$ since $\la(S)=\la(S)\da$.
\item If $r>0$ we denote by $\delta(\la(S)\coma\ca\coma r)\in \R^d$ the vector given by 
\een
\beq\label{delta r}
\delta(\la(S)\coma\ca\coma r) = 
\left( 
c_1(r)\,\uno_{s_1(r)} 
\coma 
\ldots \coma c_{q-1}(r)\, \uno_{s_{q-1}(r)-s_{q-2}(r)} \coma  
\big(\,\min\{\la_i(S)\coma c_q(r) \}\,\big)_{i=r+1}^d
\right)\ ,  
\eeq 
\ben 
\item [ ] and  $\delta(\la(S)\coma\ca\coma 0) = 
\big(\,\min\{\la_i(S)\coma c_1(0) \}\,\big)_{i\in \I_d}\,
$. 
It is easy to see (by construction) that 
\beq\label{traza bien}
\tr \delta	(\la(S)\coma\ca\coma r)  = \tr \, (S )- \tr \, (\ca)  \ .
\eeq
\een
Finally, we shall say that the index $r$ is {\bf admissible} if $r=0$ or $r>0$ and $c_{q-1}(r)<c_q(r)\,$. 
\EOE
\end{fed}

\begin{rem}\label{era f y a}
Consider a fixed strictly convex u.i.n. $N$ in $\mat$. Let $\cG_0\in\tcal$ 
be a local (or global) minimizer of 
$\Theta_{(N\coma S\coma \ca)}=\Theta:\tcal \rightarrow \R_{\geq 0}\,$.  
Assume that $k\ge d$.

\pausa
 We can apply the previous results to $\cG_0$; thus, we consider 
$p\geq 1$ and constants $c_1<\ldots<c_p$ and indexes $s_0=0<s_1<\ldots <s_p\leq d$ 
as in Theorem \ref{teo estruc min loc detallada} and Proposition \ref{pro index fea para loc min}. In particular, we get that $s_{p-1}$
is a co-feasible index which is also admissible since, if $s_{p-1}>0$, then  
$c_{p-1}<c_p$ by definition
(see Theorem \ref{teo applic1}). 
The idea of what follows is to show that 
$s_{p-1} $ (denoted $s_{p-1}(\cG_0)$ in Remark \ref{el delta con N}) 
is the {\bf unique} index which has both properties (for any norm $N$). 
First, we need to verify some properties of the vector 
$\delta(\la(S)\coma \ca \coma r)$ for a co-feasible and admissible index.
\EOE
\end{rem}


\begin{pro}\label{pro para deltaop} 
Let $S\in\matpos$ and $\ca=(a_i)_{i\in\I_k}\in(\R_{>0}^k)\da$ (with $k\ge d$). 
Let $r\in \I_{d-1}\cup\{0\}$ be a co-feasible index. Then,
with $p=q(r)$, $s_j=s_j(r)$, $c_j=c_j(r)$  for $j\in\I_p\,$, and $\delta= \delta(\la(S)\coma\ca\coma r)$  
as in Definition \ref{fed el espectro posta abstracto}, we have that: 
\ben
\item If $p>1$ then $c_1<\ldots<c_{p-1}\,$. 
\een
If we also assume that $r$ is admissible, then 
$c_{p-1}<c_p = \max\limits_{i \in \I_d}{\delta_i}$ and:
\ben
\item[2.]  $\la_{s_{p-1}+1}\ge \la_{s_p}> c_p$ and $\la_i(S) > c_j$ , for every $s_{j-1}+1\leq i\leq s_j$ and $j\in\I_{p-1}\,$. Then \beq\label{acota}
\delta_i \le \min \{ c_p \coma \la_i\} \peso{for every} i \in \I_d \ .
\eeq
\item[3.] If $p>1$ then $(a_i)_{i=s_{j-1}+1}^{s_j}\prec \big(\, \la_i(S)-c_j\big)_{i=s_{j-1}+1}^{s_j}\in \R_{>0}^{s_j-s_{j-1}}$, for 
every $j\in \I_{p-1}$. 
\item[4.] $(a_i)_{i=s_{p-1}+1}^{k}\prec \big(\,(\la_i(S)-c_p)^+\big)_{i=s_{p-1}+1}^{d}\in \R_{\geq 0}^{d-s_{p-1}}$.
\een
\end{pro}
\begin{proof}
\pausa 
1. The case $p=2$ is trivial. 
If $p>2$, assume that there exists  $j\in \I_{p-2}$ such that 
$c_j\geq c_{j+1}\, $. Then, notice that 
$$ 
P_{s_{j-1}+1\coma s_{j+1}}=\frac{s_j-s_{j-1}}{s_{j+1}-s_{j-1}}\ c_j + \frac{s_{j+1}-s_{j}}{s_{j+1}-s_{j-1}}\ c_{j+1}\leq c_j 
\ ,
$$
which contradicts the definition of $s_j$ in Eq. \eqref{sjr}, 
since $s_{j+1}\leq s_{p-1}=r$. Thus, $c_1<\ldots<c_{p-1}$. 

\pausa 
If $r=s_{p-1}$ is an 
admissible index, then
$c_{p-1}<c_p= \max\limits_{i \in \I_d}{\delta_i}$ by definition and Eq. \eqref{delta r}.
 
\pausa
2. 
By Eq. \eqref{eq para cq2}, 
we have that $c_p< \la_{i}(S)$ for $s_{p-1}+1\leq i\leq s_p\,$.
Therefore, if 
$$
j\in\I_{p-1}  \py s_{j-1}+1\leq i\leq s_j 
\implies c_j<c_p < \la_{s_{p-1}+1}(S)\leq \la_i(S) \ , 
$$
since $i\leq s_j\leq s_{p-1}<s_{p-1}+1$ and $\la(S)\in (\R^d)\da$.

\pausa
3. For $j\in\I_{p-1}$ and $s_{j-1}+1\leq m\leq s_j\,$, we have that 
\beq\label{eq majo check1}
\sum_{i=s_{j-1}+1}^ma_i\leq \sum_{i=s_{j-1}+1}^m (\la_i(S)-c_j) 
\iff c_j\leq \frac{1}{m-s_{j-1}}\ \sum_{i=s_{j-1}+1}^m\ \la_{i}(S)-a_i 
\stackrel{\eqref{Pjr}}{=}P_{s_{j-1}+1\coma m}
\eeq
(the equivalence also holds for equalities). 
Using the definition of $c_j$ (item 2. of Definition \ref {fed el espectro posta abstracto}), 
we see that the inequalities to the right in Eq. \eqref{eq majo check1} hold for every such index $m$, with equality 
for $m = s_j$ (by definition of $c_j$ and $s_j\,$).
We have proved that  $(a_i)_{i=s_{j-1}+1}^{s_j}\prec \big(\, \la_i(S)-c_j\big)_{i=s_{j-1}+1}^{s_j}
\,$. 

\pausa Item 4 follows immediately from the fact that $r=s_{p-1}$ is a co-feasible index 
(see 
Definition \ref{cofea index}).
\end{proof}

\begin {cor}\label{un cor}
Let $S\in\matpos$ and $\ca=(a_i)_{i\in\I_k}\in(\R_{>0}^k)\da$  (with $k\ge d$). 
Let $r\in \I_{d-1}\cup\{0\}$ be a co-feasible index which is also admissible. Then 
$\ca\prec  \la(S)- \delta (\la(S)\coma \ca\coma r) $. 
\end{cor}
\proof 
The relation $\ca\prec \la(S)- \delta(\la(S)\coma \ca\coma r)$  follows from 
items 3 and 4 of Proposition \ref{pro para deltaop}, since $x\prec y$ and 
$z\prec w \implies (x\coma z) \prec (y\coma w)$ (Remark \ref{desimayo}).
%
%
\QED

\begin{teo}\label{teo paren las rotativas}
Let $S\in\matpos$ and $\ca=(a_i)_{i\in\I_k}\in(\R_{>0}^k)\da$ (with $k\ge d$). 
Then there is a unique co-feasible
and admissible index $ s \in \I_{ d-1} \cup\trivial$, and this $s$ is the minimal co-feasible index.
\end{teo}
\begin{proof}
Assume that there exist two co-feasible  indexes $0\leq s<r\leq d-1$ such that 
$r$ is admissible. We show that this leads to a contradiction.
Indeed, let $s_0=0<s_1<\ldots <s_{p-1}=r<s_p\leq d$ and $c_1<\ldots <c_p$ be the indexes and constants corresponding to 
Definition \ref{fed el espectro posta abstracto}, for the index $ r$ 
(i.e., we rename $p=q(r), \ s_j = s_j(r) $ and $c_j = c_j(r)$ for $j\in \I_p$). Let $\la \igdef \la(S)\in (\R_{\geq 0}^d)\da$ and consider 
\beq\label{eq paren las rot1}
\delta=\delta(\la\coma \ca\coma r)= 
\left(c_1\, \uno_{s_1-s_0} \coma \ldots \coma c_{p-1}\,\uno_{s_{p-1}-s_{p-2}}
\coma  \big(\min\{\la_i\coma c_p\}\big)_{i=s_{p-1}+1}^d\, \right)\ .
\eeq
Similarly, consider $q = q(s)$ and 
 $s_0^*=0 <s^*_1<\ldots <s^*_{q-1}=s<s_q^*\leq d$ and $c^*_1<\ldots < c^*_{q-1}$ and $c_q^*$  be the indexes and constants corresponding to 
Definition \ref{fed el espectro posta abstracto}, for the index $s$. We also consider 
\beq\label{eq paren las rot2}
\delta^*=\delta(\la\coma \ca\coma s) =
\left(c^*_1\, \uno_{s^*_1-s^*_0} \coma \ldots\coma c^*_{q-1}\,\uno_{s^*_{q-1}-s^*_{q-2}}
\coma  \big(\min\{\la_i\coma c^*_q\}\big)_{i=s_{q-1}^*+1}^d\, \right)\ .
\eeq
If $\delta^*= \delta$ then by Eqs.  \eqref{eq para cq2}, \eqref{eq paren las rot1} 
and \eqref{eq paren las rot2}, 
$s_{q}^*= s_p = \max \{i \in \I_d : \delta_i < \la_i(S)\}$, and 
$c_{q}^*= c_p =\delta_{s_{p}}\,$. But in this case  
$r = s_{p-1}= \min\{i \in \I_{d-1} :\delta_{i+1} = c_{p}\} \le s_{q-1}^*= s$, a contradiction. Hence 
$\delta^*\neq \delta$. 

\pausa
{\bf Case 1.} 
Assume that there exists $1\leq j\leq \min\{p-1\coma q-1\}$ such that 
$$
\text{$s_i=s^*_i$ \ \ (and then also $c_i=c^*_i$) \ \  for \ $0\leq i\leq j-1$ \ , \ \   but\ \   $s_j\neq s^*_j$ .}
$$
Next we show that this leads to a contradiction ($\delta^*= \delta$). Indeed, 
since $s_{j-1} = s_{j-1}^*\,$, by construction
$$
s_j=\max\{s_{j-1}< i\leq r 
: \ P_{s_{j-1}+1\coma i}=\min_{s_{j-1}+1\leq \ell \leq r }
P_{s_{j-1}+1\coma \ell} \} \peso{with} c_j=P_{s_{j-1}+1\coma s_j}
$$
and 
$$
s_j^*=\max\{s_{j-1}< i\leq s 
: \ P_{s_{j-1}+1\coma i}=\min_{s_{j-1}+1\leq \ell \leq s}
P_{s_{j-1}+1\coma \ell} \} \peso{with}
 c^*_j=P_{s_{j-1}+1\coma s^*_j}\ .
$$
Using that the limits $s<r$, then $\min\limits_{s_{j-1}+1\leq \ell \leq s} P_{s_{j-1}+1\coma \ell} \ge 
\min\limits_{s_{j-1}+1\leq \ell \leq r} P_{s_{j-1}+1\coma \ell}\,$. 
Since $s_j^* \neq s_j\,$, this fact easily shows 
that 
\beq\label{cosas}
c_j\leq c^*_j \py 
s_{q-1}^*=s <s_j\leq r  \ .
\eeq 
On the other hand, by Corollary \ref{cor sobre constantes cofeas} we have that 
$c^*_q=c_q(s)\geq c_p(r)= c_p\,$, since they 
are the co-feasible constants corresponding to the co-feasible indexes $s^*_{q-1}=s<r=s_{p-1}\,$.

\pausa
With these facts we can compare $\delta $ and $\delta^*$: 
\bit
\item  We have that 
$\delta_i=\delta_i^*$ for $1\leq i\leq s_{j-1}= s_{j-1}^*$ 
by hypothesis. 
\item By Eq. \eqref{eq paren las rot1},  \eqref{eq paren las rot2}, 
and  item 1 of Proposition \ref{pro para deltaop} ($c_j^*< \dots<c_{q-1}^*$), 
$$
\delta_i^*\ge c_j^*\stackrel{\eqref{cosas}}{\geq}  c_j = \delta_i 
\peso{for} s_{j-1} = s_{j-1}^*< i\leq s^*_{q-1} \stackrel{\eqref{cosas}}{<}
s_j \ .
$$
\item 
Since $c_p = \max\{\delta_j: j\in \I_{d}\}$ by 
Proposition \ref{pro para deltaop}  ($r$ is admissible),  then 
$$
\delta_i^*= c_q^* \ge c_p 
\ge\delta_i  
\peso{for} s^*_{q-1} <i \le s_q^* 
\ .
$$ 
\item 
Finally,  $\delta_i^*=\la_i\geq \delta_i\,$, for $s^*_q< i\leq d$ 
(item 2 in Proposition \ref{pro para deltaop}). 
\eit
Therefore $\delta\leqp \delta^*$. Since
$\tr(\delta)=\tr(S)-\tr(\ca)=\tr(\delta^*)$ by Eq. \eqref{traza bien}, 
we get that $\delta=\delta^*$, a contradiction. 

\pausa
{\bf Case 2.} If we assume that  $p \le q$ and $s_j=s^*_j$ (and hence $c_j=c^*_j$) for $0\leq j\leq p-1$, 
then 
$$
s_{p-1}=s^*_{p-1}\leq s^*_{q-1}=s<r=s_{p-1}  \ .
$$
{\bf Case 3.} Finally, if 
$q< p$ and $s_j=s^*_j$ (and hence $c_j=c^*_j$) for $0\leq j\leq q-1$, 
then we have that $\delta_i=\delta^*_i$ for $1\leq i\leq s^*_{q-1}=s_{q-1}$. 
Then, by 
Proposition \ref{pro para deltaop}, we have that  
$$
c_p \le c^*_q  \implies 
\delta_i \stackrel{\eqref{acota}}{\le} \min\{c_p\coma \la_i\}\leq 
\min\{c^*_q\coma \la_i\} \stackrel{\eqref{eq paren las rot2}}{=} \delta^*_i\peso{for} s^*_{q-1}< i\leq d\ .
$$
Hence, $\delta\leq \delta^*$. Using that $\tr(\delta)=\tr(\delta^*)$, also in this case  we conclude that $\delta=\delta^*$.
The proof finishes once we notice that one of these three cases should occur. 
\end{proof}

\begin{fed}\label{def espec posta}\rm
Let $S\in\matpos$ with $\la=\la(S)\in (\R_{\geq 0}^d)\da$ and $\ca=(a_i)_{i\in\I_k}\in (\R_{>0}^k)\da$ (with $k\geq d$). 
If $ s \in \I_{ d-1} \cup\trivial$ is the unique
co-feasible and admissible index for $S$ and $\ca$ (which exists by Remark \ref{era f y a}), then we denote by 
 $\delta(\la\coma\ca)\igdef\delta(\la\coma \ca\coma s)$ as in Eq. \eqref {delta r} of 
Definition 
\ref{fed el espectro posta abstracto}.\EOE
\end{fed}

\begin{rem}\label{rem es computable}
With the notation of Definition \ref{def espec posta} above, notice that the vector $\delta(\la\coma \ca)$ can be computed using 
a fast algorithm. Indeed, the notion of co-feasible and admissible index is algorithmic and can be checked using 
a fast routine; once the unique co-feasible and admissible index is computed, the vector $\delta(\la\coma\ca)$ can also be computed 
using a fast algorithm (Definition \ref{fed el espectro posta abstracto}).\EOE
\end{rem}

\begin{teo}\label{teo main1}
Let $S\in\matpos$ with $\la=\la(S)\in (\R_{\geq 0}^d)\da$, $\ca=(a_i)_{i\in\I_k}\in (\R_{>0}^k)\da$ (with $k\geq d$) and 
$\delta(\la\coma \ca)$ as in Definition \ref{def espec posta}. If $N$ is a strictly convex u.i.n. in $\mat$ and 
$\cG_0\in\tcal$
then, the following statements are equivalent:
\ben
\item $\cG_0\in\tcal$ is a global minimizer of $\tnsa$;
\item $\cG_0\in\tcal$ is a local minimizer of $\tnsa$; 
\item $\la(S-S_{\cG_0})=\delta(\la\coma\ca)\da$.
\een
Hence, the global (and local)  minimizers are the same for every strictly convex u.i.n. $N$. 
\end{teo}
\begin{proof}
Clearly, $1.\Rightarrow 2.$ In order to see $2.\Rightarrow 3.$, we recall Remarks  \ref{el delta}, \ref{el delta con N} 
and \ref {era f y a}, where we have seen that $\la(S-S_{\cG_0})=\delta(S\coma\ca \coma \cG_0)\da$, 
for the vector 
$\delta(\la\coma\ca \coma \cG_0)$ given in Eq. \eqref{del SaN} and 
completely determined by the index called $s_{p-1}(\cG_0)$. By Remark \ref{era f y a} and 
Theorem \ref{teo paren las rotativas}, this  $s_{p-1}(\cG_0)$ is 
the unique  co-feasible and admissible index of Theorem \ref{teo paren las rotativas}. Therefore, by Equations \eqref{del SaN} and \eqref{delta r}, 
$$
\delta (\la(S)\coma \ca) = \delta(S\coma\ca \coma \cG_0) \implies \la(S-S_{\cG_0})=\delta(\la\coma\ca)\da \ .
$$

\pausa
$3.\Rightarrow 1.$ 
Notice that $\Theta$ is a continuous function defined on a compact metric space, so then there exists 
$\cG_1\in\tcal$ that is a global minimizer of $\Theta$ and, in particular, a local minimizer.  
By the already proved $2.\Rightarrow 3.$, we must have that $\la(S-S_{\cG_1})=\delta(\la\coma\ca)\da = \la(S-S_{\cG_0})$.

\pausa
In particular, since $N$ is unitarily invariant 
$$
\tnsa(\cG_0)=N(S-S_{\cG_0})=N(D_{\delta(\la\coma\ca)})= N(S-S_{\cG_1})=\tnsa(\cG_1)\  ,
$$ 
where $D_{\delta(\la\coma\ca)}\in\mat$ denotes the diagonal matrix with main diagonal $\delta(\la\coma\ca)$.
\end{proof}

\pausa
We end this section with the following examples.

\begin{exa} Consider $\cB=\{e_1\coma e_2\}$ the canonical basis of $\C^2$. Let 
$S=3\,e_1\otimes e_1+e_2\otimes e_2\in \matrec{2}^+$ and $\ca=(1\coma 1)$ 
(i.e. $k=d=2$).
Then $S$ is an invertible operator. Consider the vectors $g_1=g_2=e_1$, 
and  $\cG_0=\{g_1\coma g_2\}\in \mathbb T_2(\ca)$. Then $\la(S-S_{\cG_0})= 
\la(e_1\otimes e_1+e_2\otimes e_2) =(1\coma 1)$. 
If $\cG\in \mathbb T_2(\ca)$ is arbitrary, then
$\tr \la(S-S_\cG)= \tr\,S -\tr\, S_{\cG}  = 2$. Hence 
$$\la(S-S_{\cG_0})=(1\coma 1)\prec \la(S-S_\cG)
\implies s(S-S_{\cG_0})=(1\coma 1)\prec_w s(S-S_\cG) \ , 
$$ 
by Remark \ref{desimayo} and Theorem \ref{teo intro prelims mayo}. 
Then $\tnsa(\cG_0)\leq \tnsa(\cG)$, for every u.i.n. $N$. 
Thus, $\cG_0=\{e_1\coma e_1\}$ is a global minimizer of $\tnsa$ in $\mathbb T_2(\ca)$.
Therefore this problem is co-feasible, so that $p=1$, $s_1 =\rk\, S_{\cG_0}= 1$ and $c_1 = \la_2(S) = 1$. 
Notice that in this case $\cG_0$ is not a frame for $\C^2$ (even when $S\in\mathcal M_2(\C)^+$ is invertible and $k\geq d$). \EOE	
\end{exa}

\begin{exa}
Consider $\cB=\{e_1\coma e_2\}$ the canonical basis of $\C^2$. Let  $S=e_1\otimes e_1
\in \matrec{2}^+$  and $\ca=(2,1)$ (with $k=d=2$ again).
Then $S$ is a non-invertible operator.  
We shall see that 
$\cG_0=\{2\,e_1\coma e_2\}\in\mathbb T_2(\ca)$ is a global minimizer of $\tnsa$, 
for every u.i.n. $N$.
Indeed, 
$$
\la (S-S_{\cG_0}) = \la(-e_1\otimes e_1-e_2\otimes e_2) =(-1\coma -1) \implies s(S-S_{\cG_0})= |\la(S-S_{\cG_0})|=(1,1) 
$$ 
and, if $\cG\in \mathbb T_2(\ca)$ 
is arbitrary, then $\tr\, \la (S-S_{\cG}) =1-3=-2$, so that
 $\tr \, s(S-S_{\cG})\geq 2$. This last fact implies that 
$s(S-S_{\cG_0})\prec_ws(S-S_{\cG})$ and therefore $\tnsa(\cG_0)\leq \tnsa(\cG)$. 
Also this problem is co-feasible, with $p=1$, $s_1 =\rk\, S_{\cG_0}= 2$ 
and $c_1  =- 1$. 
Notice that in this case $\cG_0$ is a frame 
for $\C^2$ (even when $S\in\mathcal M_2(\C)^+$ is not an invertible operator).\EOE
\end{exa}
\subsection{The general case}\label{sec kmenord}

So far, we have considered the case of local minimizers of GFOD functions when the number of vectors $k$ is greater than or equal to
the dimension of the space $d$. This was essentially needed in Section \ref{sec cofea}.
In this section we add the case when $k<d$, thus covering all possible cases. Our approach is based on a reduction to the case considered in Section \ref{sec 4.1.}.

\begin{fed}\label{defi delta op kmenord} \rm
Let $S\in\matpos$ and let $\ca=(a_i)_{i\in\I_k}\in(\R_{>0}^k)\da$ with $k<d$.
Let $\cB=\{v_i\}_{i\in\I_d}$ be an ONB of $\C^d$ such that  
$S=\sum_{i\in\I_d}\la_i(S)\ v_i\otimes v_i\, $. Let 
$$
V_k=\text{span}\{v_i : i \in \I_k\}
\peso{and} S_k \igdef S|_{V_k}
=\sum_{i\in\I_k}\la_i(S)\ v_i\otimes v_i \in L(V_k)^+\ .
$$ 
Since $k = \dim V_k$ (the ``new $d$") we can take 
$\delta(\la(S_k)\coma \ca)\in\R^k$ using Definition \ref{def espec posta}, 
for the data $\la(S_k) = (\la_1(S)\coma \dots \coma \la_k(S)\, ) \in (\R_{\ge0}^k)\da$ 
and $\ca\in(\R_{>0}^k)\da$. 
We define the vector 
$$
\delta(\la(S)\coma\ca) \igdef \big(\, \delta(\la(S_k)\coma \ca)\coma 
 \la_{k+1} (S)\coma 
\dots \coma \la_d(S) 
\, \big)\ , 
$$
which does not really depends on $S_k$ and $\cB$, but only on $\la (S) $ and $\ca$. 
\EOE
\end{fed}

\begin{teo}\label{si kmenord}
Let $S\in\matpos$, let $\ca=(a_i)_{i\in\I_k}\in(\R_{>0}^k)\da$ 
and let $N$ be a strictly convex u.i.n. in $\mat$.
 Given $\cG_0=\{g_i\}_{i\in\I_k}\in\tcal$ the following are equivalent:
\ben
\item $\cG_0$ is a global minimizer of $\Theta_{(N\coma S\coma \ca)}$;
\item $\cG_0$ is a local minimizer of minimizer of $\Theta_{(N\coma S\coma \ca)}$;
\item $\la(S-S_{\cG_0})=\delta(\la(S)\coma\ca)\da$ 
(see Definition \ref{def espec posta} if $k\ge d$, and 
Definition \ref{defi delta op kmenord} if $k<d$). 
\een
\end{teo}
\begin{proof}
If $k\ge d$ this is Theorem \ref{teo main1}. Let us assume that $k<d$. 

\pausa
Clearly $1.\Rightarrow 2$. If we assume $2$ we can apply Theorem \ref{teo applic1}, Proposition \ref{nueva pro 3.10 y 3.11}
and Theorem \ref{teo estruc min loc detallada} 
(these statements do not assume that $k\ge d$).   
With the notation of these results (i.e., with Notation \ref {nuevas notas 2}), 
there exists $\mathcal B=\{v_i\}_{i\in\I_d}$ an ONB of $\C^d$ such that 
$$
S=\sum_{i\in\I_d} \la_i(S) \ v_i\otimes v_i
\py  S_0=S_{\cG_0} = \sum_{i\in\I_d} \la_i(S_0) \ v_i\otimes v_i  
\,.
$$ 
We have that 
$r\igdef  \rk\, S_0\le k$, and $W=R(S_{\cG_0}) = 
\text{span}\{v_i\}_{i\in\I_r} \inc \text{span}\{v_i\}_{i\in\I_k}= V_k\,$,
as in Definition \ref{defi delta op kmenord}. 
Since $\la_i (S_0) = 0 $ for $i > k$, the vector  
$\delta\igdef \big(\,\la_i(S)-\la_i(S_0)\,\big)_{i\in\I_k}\in \R^k$ 
satisfies that 
\beq \label{trunc}
\la(S-S_{0})= \big( \,\la(S)-\la(S_{0})\, \big)\da \py \la(S)-\la(S_{0})  
=\big( \,\delta \coma \la_{k+1} (S)\coma \dots \coma \la_d(S)\, \big)\ .
\eeq
With the notation of Definition \ref{defi delta op kmenord}, we have to prove that 
$\delta = \delta(\la(S_k)\coma \ca)$. Since $r=s_p\le k <d$, we can apply 
 Remark \ref{el delta} (to $\la(S)-\la(S_{0})  \in \R^d$), so that 
\beq\label{cas1}
\delta = 
\big(\, c_1 \, \uno_{s_1} \coma c_2 \, \uno_{s_2-s_{1}} \coma 
\dots \coma c_p\, \uno_{s_p-s_{p-1}} \coma \la_{s_p+1} (S)\coma 
\dots \coma \la_k(S)\,\big) \peso{if} s_p<k\,
\eeq
or 
\beq\label{cas2}
\delta  = \big(\, c_1 \, \uno_{s_1} \coma c_2 \, \uno_{s_2-s_{1}} \coma 
\dots \coma c_p\, \uno_{s_p-s_{p-1}} \,\big)\peso{if} s_p=k\ ,
\eeq
where the indexes $s_1<\ldots<s_{p-2}$ and constants $c_1<\ldots c_{p-1}$ 
are constructed (for $\la(S-S_{\cG_0})$ and therefore also for $\delta$) 
in terms of the index $s_{p-1}\,$ (when $p>1$) using the algorithm given in   
Theorem \ref{teo estruc min loc detallada} (and also 
in Definition \ref{fed el espectro posta abstracto}, with respect to 
$\la(S_k)$, $\ca$ and $s_{p-1}$).
Also $c_{p-1}<c_p$ 
by Theorem \ref{teo applic1}.   

\pausa
Therefore, in order to show that $\delta = \delta(\la(S_k)\coma \ca)$, 
by Theorem \ref{teo paren las rotativas}  we just need to prove that
the index $s_{p-1}\in\I_{k-1}\cup\{0\}$ is co-feasible (and admissible) with respect to $S_k$ and $\ca$. 
By  Theorems \ref{teo applic1} and \ref{nueva pro 3.10 y 3.11} we know that
$(S-S_0)g_i = c_p \, g_i \iff s_{p-1}+1\leq i\leq k$, and 
$$
W_p =\gen\{g_i : \, s_{p-1}+1\leq i\leq k\ \} = \gen\{v_i : \, s_{p-1}+1\leq i\leq s_p\ \} \ .
 $$
Hence, if we let $X=
\gen\{v_i:\, s_{p-1}+1\leq i\leq k\ \}$ and 
$\cG_p=\{g_i\}_{i=s_{p-1}+1}^k\in\mathbb T_{X}(\ca^{(s_{p-1})})$
 then $$(S_k|_X-S_{\cG_p}) \, g_i= (S-S_0)\, g_i = c_p\, g_i\ , \peso{for} s_{p-1}+1\leq i\leq k\ .$$
By Remark \ref{R3.15} (for $S_k$ and $\ca$), 
we only need to show that $c_p =\max\limits_{s_{p-1}+1\leq i \leq k} \ \delta_i \ (=  \max\limits_{i \in \I_k} \ \delta_i \ )\,$. 

\pausa
Suppose that $c_p < \max \sigma (S-S_0)$. Then, by item 5 of Theorem \ref{teo applic1}, 
the set $\cG_0$ is linearly independent 
(since each set $\{g_j\}_{j\in J_j}$ is linearly independent, and 
they are  sets of  eigenvectors of the different eigenvalues $c_j\,$). Then
$s_p = \rk \, S_0 = k$, so we can apply Eq. \eqref{cas2}, and 
automatically $c_p = \max\limits_{i \in \I_k} \ \delta_i \,$. 

\pausa
Otherwise we have that  $c_p = \max \sigma (S-S_0) \ge \max\limits_{i \in \I_k} \ \delta_i \,$.  
Then, in any case $c_p = \max\limits_{i \in \I_k} \ \delta_i \,$. We have proved that 
the index $s_{p-1}$ is co-feasible (and also admissible, because $c_{p-1}<c_p$) 
with respect to $S_k$ and $\ca$. 
Then $\delta = \delta(\la(S_k)\coma \ca)$ by Theorem \ref{teo paren las rotativas} and 
$\la(S-S_{\cG_0})=\delta(\la(S)\coma\ca)\da$ by Eq. \eqref{trunc}.

\pausa
$3. \Rightarrow 1$. An argument analogous to that in the proof of Theorem \ref{teo main1} ($3. \Rightarrow 1.$) proves this implication.
\end{proof}

\begin{rem} \rm
The proof of $2. \Rightarrow 3.$ of Theorem \ref{si kmenord} becomes trivial if we assume 
that (the vectorial version of) the norm $N$ satisfies that, for $x\coma y\in \R^k$ and $z\in \R^{d-k}$, 
\beq\label{nuimala}
N(x\coma z) \le N(y\coma z) \implies N(x\coma 0) \le N(y\coma 0) \ ,
\eeq
since in this case $\cG_0$ is still a local minimizer for $S_k$ and $\ca$ in $V_k\,$.
The most usual strictly convex norms (for example $p$-norms, for $p\in (1,\infty)$) satisfy 
Eq. \eqref{nuimala}, but this property fails in general.  Take $N = \|\cdot \|_\infty + \|\cdot \|_2\,$
which is a strictly convex UIN. In this case, if $r = \frac{\sqrt{2}}{2}\ $, then ($d=3,\, k=2$)
\beq
N\big(\,(0\coma 1)\coma 1\big)=1+\sqrt {2}\  
=	 N\big(\,(r\coma r)\coma 1 \big)
\ \ \text{but} \ \ 
  N\big(\,(0\coma 1)\coma 0\big) =2 >r+1 =N\big(\,(r\coma r) \coma 0 \big) \ .
\EOEP
\eeq
\end{rem}

\begin{cor}\label{cor es min mayo} 
With the notation of Theorem 
\ref{si kmenord}, we have that 
$$
|\delta(\la\coma \ca)|\prec_w |\la(S-S_{\cG})|  \ \ , \peso{for every} \cG\in\tcal  \ .
$$
\end{cor}
\begin{proof}
For $h\in\I_d$ and $\varepsilon>0$ let 
$$
N_{(h\coma \varepsilon)}(A)= 
N_{(h)}(A) +\varepsilon\, \|A\|_2 = \sum_{i\in\I_h} s_i(A)+\varepsilon\, \|A\|_2 \ , \peso{for} A\in\mat \ .
$$
Then, $N_{(h\coma \varepsilon)}$ is a strictly convex u.i.n. in $\mat$ such 
that $\lim\limits_{\varepsilon\rightarrow 0^+}N_{(h\coma \varepsilon)}(A)=
N_{(h)}(A)$, 
 for $A\in\mat$.
If we let $\cG_0\in\tcal$ be such that $\la(S-S_{\cG_0})=\delta(\la\coma\ca)\da$ then, by Theorem \ref{si kmenord},   
$$
\barr{rl}
\sum_{i\in\I_h}|\delta(\la\coma \ca)|\da_i & 
= \displaystyle N_{(h)}(S-S_{\cG_0})
=\lim_{\varepsilon\rightarrow 0^+}N_{(h\coma \varepsilon)}(S-S_{\cG_0}) \\&\\
& \leq \displaystyle 
\lim_{\varepsilon\rightarrow 0^+}N_{(h\coma \varepsilon)}(S-S_{\cG})
=\sum_{i\in\I_h}|\la(S-S_\cG)|_i\da
\ .
\earr
$$
Since this occurs for every $h\in\I_d\,$, then $|\delta(\la\coma \ca)|\prec_w |\la(S-S_{\cG})|$. 
\end{proof}

\section{Proof of some technical results}\label{Appendixity}

In this section we prove some results stated in Section \ref{sec inner struc}.
We begin by re-stating Notation \ref{nuevas notas 2}, 
 that we will use again
throughout this section.

\pausa
{\bf Notation \ref{nuevas notas 2} (repeated).}  
Fix $S\in\matpos$,  $\asubi$, 
and a strictly convex u.i.n. $N$ on $\mat$. Also consider 
the notions introduced in Theorem \ref{teo applic1}. As before, let 
\begin{enumerate}
\item $\tnsa=\Theta: \tcal \rightarrow \R_{\geq0}$ given by $\Theta(\cG)=N(S-S_{\cG})$.
\item A local minimizer $\cG_0=\llav{g_i}_{i\in\I_k}\in\tcal$ of $\tnsa$, with frame operator 
$S_0=S_{\cG_0}$.
\item We denote by  $\la=(\la_i)_{i\in\I_d} = \la(S) \in (\R_{\geq 0}^d)\da$  and 
$\mu=(\mu_i)_{i\in\I_d} = \la(S_0) \in (\R_{\geq 0}^d)\da$.

\item We fix $\cB=\{v_i\}_{i\in\I_d}$ an ONB of $\C^d$ as in Theorem \ref{teo applic1}. Hence, 
$$
S=\sum_{i\in\I_d} \la_i \ v_i\otimes v_i
\py S_{0}=\sum_{i\in\I_d} \mu_i \ v_i\otimes v_i  \ ,
$$

\item We consider
$W=R(S_0)$, $D=(S-S_0)|_W$ and 
$\sigma(D)=\{c_1\coma\ldots\coma c_p\}$ where $c_1<c_2<\ldots<c_p$. 

\item Let $ s_D= \max \, \{i\in\I_d: \mu_i \neq 0\} = \rk \, S_{0}\,$.

\item We denote by $  \delta= \la - \mu \in \R^d$ 
so that 
$$
S-S_0=\sum_{i\in\I_d} \delta_i \ v_i\otimes v_i \peso{and}
D=\sum_{i=1}^{s_D} \delta_i \ v_i\otimes v_i\ .
$$
Notice that $\delta$ is constructed by pairing the entries 
of ordered vectors (since $\la=\la(S)$ and $\mu=\la(S_0)$\,
Nevertheless, we have that $\la(S-S_0) = \delta\da$. 
In what follows we obtain some properties of (the unordered vector) $\delta$. 

\item For each $j\in \I_p\,$, we consider the following sets of indexes: 
$$K_j = \{ i \in \I_{s_D} :   \delta_i=\la_i -\mu_i = c_j\}  
\py J_j = \{i\in \I_k: D\,g_i = c_j \, g_i\}  \ .
$$ 
Theorem \ref{teo applic1} assures that   
$\I_{s_D}  
= \stackrel{\text{D \ \ \ \ \ }}{\bigcup_{j\in\I_p}} \ K_j \py \I_k  
= \stackrel{\text{D \ \ \ \ \ }}{\bigcup_{j\in\I_p}} \ J_j$ (disjoint unions).
\item By Eq. \eqref{el SF}, $R(S_{0})= \gen\{g_i : i \in \I_k\} = W = 
\bigoplus_{i\in\I_p} \ker \,(D-c_i\,I_W\,)$
then, for every $ j\in \I_p\ $, 
\beq\label{cajas2}
W_j =\gen\{g_i : i \in J_j\} = \ker \,(D-c_j\,I_W\,) = \gen\{v_i : i \in K_j\} \ ,
\eeq
because $g_i \in  \ker \,(D-c_j\,I_W\,)$ for every $i \in J_j\,$. 
Note that, by Theorem \ref{teo applic1}, each $W_j$ reduces both $S$ and $S_{0}\,$. \EOE 
\end{enumerate}

\pausa
In order to prove Proposition \ref{nueva pro 3.10 y 3.11} we first present the following two results.

\begin{pro}\label{pro apend 1} Let $S\in\matpos$ and let $\cG_0\in \tcal$ be as in Notation \ref{nuevas notas 2} and assume that $p>1$. 
Assume that there exist 
\beq\label{eq nueva pro a1}
i< r\leq p \ \ ,  \ \ h\in J_i \ \ ,  \ \  l\in J_r \peso{with} l<h \quad  (\then a_l\geq a_h)\,.
\eeq
 Then, there exists a
continuous curve $\cG(t):[0,1)\rightarrow \tcal$ such that $\cG(0)=\cG_0$ and 
$\la(S-S_{\cG(t)})\prec \la(S-S_0)$ with strict majorization for $t\in (0,\varepsilon)$ for some $\varepsilon >0$.
\end{pro}
\begin{proof}
Consider 
$$w_h=g_h/\norm{g_h}= a_h^{-1/2} g_h \py w_l=g_l/\norm{g_l}= a_l^{-1/2} g_l,$$ 
(note that $\pint{w_h,w_l}=0$ because
$\pint{g_h,g_l}=0$). Now define, for $t\in \R$ and for some convenient $\gamma\in \R\setminus\{0\}$  (which will be explicitly calculated later),
$$
g_h(t)= \cos(t)\, g_h + \sin(t) \,\norm{g_h}\, w_l
\py
g_l(t)= \cos(\gamma t)\, g_l + \sin(\gamma t) \,\norm{g_l} \,w_h.
$$ 
Then consider the family $\cG_{\gamma}(t)$, which is obtained from $\cG_0$ by 
replacing the vectors $g_h$ and $g_l$ by $g_h(t)$ and $g_l(t)$ respectively, and denote by $S_{\gamma}(t)$ its frame operator. 
Note that $\cG_\gamma(t)\in\tcal$ for every $t\in\R$ and $\cG_\gamma(0)=\cG_0$.

\pausa
Let $W_{h,l}=\gen\llav{w_h,w_l}$, this subspace reduce 
both $S-S_0$ and $S-S_{\gamma}(t)$. The fact that $g_h(t), g_l(t)\in W_{h,l}$, allows us to 
represent the following matrix with respect to the basis $\llav{w_h,w_l}$ of $W_{h,l}$,
$$
g_h\otimes g_h=
\begin{pmatrix}
a_h & 0\\
0 & 0
\end{pmatrix},\,\,\,
g_h(t)\otimes g_h(t)= a_h
\begin{pmatrix}
\cos^2(t) & \cos(t)\sin(t)\\
\cos(t)\sin(t) & \sin^2(t)
\end{pmatrix},
$$
$$
g_l\otimes g_l=
\begin{pmatrix}
0 & 0\\
0 & a_l 
\end{pmatrix},\,\,\,
g_l(t)\otimes g_l(t)= a_l
\begin{pmatrix}
\sin^2(\gamma t) & \cos(\gamma t)\sin(\gamma t)\\
\cos(\gamma t)\sin(\gamma t) & \cos^2(\gamma t)
\end{pmatrix}.
$$
Then,
$$
S-S_{\gamma}(t) = S-S_0-g_h(t)\otimes g_h(t) -g_l(t)\otimes g_l(t) + g_h\otimes g_h + g_l\otimes g_l.
$$
Hence $(S-S_{0})|_{W_{h,l}^{\perp}}=(S-S_{\gamma}(t))|_{W_{h,l}^\perp}$. On the other hand
$
(S-S_{0})|_{W_{h,l}}=
\begin{pmatrix}
c_i & 0\\
0 & c_r
\end{pmatrix}
$
and
$$
(S-S_{\gamma}(t))|_{W_{h,l}}=
\begin{pmatrix}
c_i + a_h \sin^2(t) - a_l \sin^2(\gamma t) & - a_h\cos( t)\sin( t)-a_l\cos(\gamma t)\sin(\gamma t)\\
- a_h\cos( t)\sin( t)-a_l\cos(\gamma t)\sin(\gamma t)& c_r + a_l \sin^2(\gamma t) - a_h \sin^2(t)
\end{pmatrix}
\stackrel{def}{=} A_\gamma (t).
$$
Since $\tr(A_\gamma (t))=c_i+c_r$ for every $t\in\R$, then we have the strict majorization
$\la(A_\gamma (t))\prec (c_r,c_i)$ if and only if $\norm{A_\gamma (t)}^2_2<c_r^2+c_i^2$.
So consider the function $m_\gamma:\R\to\R$ given by 
$$m_\gamma(t)=\norm{A_\gamma (t)}^2_2=\tr(A_\gamma (t)^2)\,\,\,(\forall t\in\R).$$
Notice that $A_\gamma(0)= (S-S_0)|_{W_{h,l}}$, then $m_\gamma(0)=\tr((S-S_0)|_{W_{h,l}}^2)=c_r^2+c_i^2$. 
The next step is to find a convenient $\gamma\in \R\setminus \{0\}$ such that $m_\gamma^\prime(0)=0$ but $m_\gamma^{\prime\prime}(0)<0$; 
in this case we obtain 
the strict 
majorization $\la(A_\gamma (t))\prec (c_r,c_i)$ for $t\in(0,\varepsilon)$, for some $\varepsilon>0$.
This last fact implies that 
$\la(S-S_\gamma(t))\prec \la(S-S_0)$ strictly, for $t\in(0,\varepsilon)$, as desired.

\pausa
Start computing the derivatives of the entries $a_{ij}(t)$ of $A_\gamma(t)$, for $1\leq i,j\leq 2$:
\begin{align*}
a_{11}^\prime(t)&= a_h\sin(2\,t)-a_l\gamma\sin(2\,\gamma t) &\then &\  a_{11}^\prime(0)=0, \\
a_{12}^\prime(t)&= -a_h\cos(2\,t)-a_l\gamma \cos(2\,\gamma t)=a_{21}^\prime(0)&\then&\ a_{12}^\prime(0)=-a_h-a_l\gamma,\\
a_{22}^\prime(t)&= a_l\gamma\sin(2\,\gamma t)-a_h \sin(2\,t) &\then& \ a_{22}^\prime(0)=0,\\
a_{11}^{\prime\prime}(t)&= 2\, a_h \cos(2\,t)-2 \,a_l\gamma^2 \cos(2\,\gamma t)) &\then& \ a_{11}^{\prime\prime}(0)=2 \,(a_h- a_l\gamma^2), \\
a_{12}^{\prime\prime}(t)&= 2\, a_h \sin(2\,t) + 2\, a_l\,\gamma^2 \sin(2\,\gamma t) &\then&\ a_{12}^{\prime\prime}(0)=0,\\
a_{22}^{\prime\prime}(t)&= 2 \,a_l\gamma^2 \cos(2\,\gamma t)-2 a_h \cos(2\,t) &\then&\ a_{22}^{\prime\prime}(0)=2\,(a_l\gamma^2-a_h).
\end{align*}
Then
\begin{align*}
m_\gamma^\prime (0)&=2 \, a_{11}(0)\, a_{11}^\prime(0)+4\,a_{12}(0)\, a_{12}^\prime(0)+2 \, a_{22}(0)\, a_{22}^\prime(0)=0\, ,\\
m_\gamma^{\prime\prime} (0)&= 2\, a_{11}(0)\, a_{11}''(0)+ 4\, (a_{12}'(0))^2+ 2\, a_{22}(0)\, a_{22}''(0)
\\ &=4\, c_i(a_h-a_l\,\gamma^2)+4 (a_h+a_l\gamma)^2+4\, c_r(a_l\,\gamma^2 -a_h).
\end{align*}
Note that $m_\gamma^{\prime\prime} (0)$ is a quadratic function depending on $\gamma$ whose discriminant is
$$
a_h^2\, a_l^2\corch{a_h\, a_l- (a_l+c_r-c_i)(a_h+c_i-c_r)}>0,
$$
because we assume that $a_h\leq a_l$ (and we have that $c_r>c_i$),
$$
(a_l+(c_r-c_i))(a_h-(c_r-c_i))=  a_l\, a_h+ (c_r-c_i)(a_h-a_l)-(c_r - c_i)^2<a_l\, a_h.
$$
Then, there exists $\gamma\in\R\setminus\{0\}$ such that $m_\gamma^{\prime\prime}(0)<0$. 
\end{proof}

\pausa
The following result together with Proposition \ref{pro apend 1} will allow us to obtain a proof of Proposition \ref{nueva pro 3.10 y 3.11} (see below).

\begin{pro}\label{pro apend 2} Let $S\in\matpos$ and let $\cG_0\in \tcal$ be as in Notation \ref{nuevas notas 2} and assume that $p>1$.
Assume that there exist 
\beq\label{eq no vale estru kj}
i\in K_e \py j\in K_r \peso{with} e< r \peso{such that} j<i\,.\eeq
 In this case, we construct a 
continuous curve $\cG(t):[0,1)\rightarrow \tcal$ such that $\cG(0)=\cG_0$ and such that 
$\la(S-S_{\cG(t)})\prec \la(S-S_0)$ with strict majorization for $t\in (0,\varepsilon)$ for some $\varepsilon >0$.
\end{pro}
\begin{proof} 
With the notation of the statement and Notation \ref{nuevas notas 2}, notice that
$$
\mu_i\leq \mu_j     \py  {c_e=\la_i-\mu_i < c_r=\la_j-\mu_j} \ . 
$$
As in 
Notation \ref{nuevas notas 2}, consider $\cB=\{v_l\}_{l\in\I_d}$ an ONB of $\C^d$ such that 
\beq
S=\sum_{\ell\in\I_d} \la_\ell \ v_\ell\otimes v_\ell
\py S_{0}=\sum_{\ell\in\I_d} \mu_\ell \ v_\ell\otimes v_\ell  \ .
\eeq 
For $t\in[0,1)$ we let 
\beq\label{defi glt}
 g_l(t)=  g_l + \big(\, (1-t^2)^{1/2}-1\,\big)\ \langle g_l,v_i\rangle \ v_i + t \ \langle g_l,v_i\rangle \ v_j \peso{for} l\in \I_k\,. 
\eeq
Notice that, if $l \in J_e\,$, then  
$(S-S_0) \, g_l = c_e\, g_l \implies \api g_l\coma v_j\cpi =0$. 
Similarly, if $l\in \I_k\setminus J_e$ then $\langle g_l\coma v_i\rangle =0$ (so that  $ g_l(t)=g_l$). 
Therefore the sequence   $\cG(t)=\{g_l(t)\}_{l\in\I_k}\in\tcal $ for $t\in[0,1)$. 
Let $P_i=v_i\otimes v_i$ and $P_{ji}=v_j\otimes v_i$ (so that $P_{ji}\,x=\langle x\coma v_i\rangle \ v_j$). Then, 
for every $t \in [0\coma 1)$,  
$$
g_l(t)= \big(\,I+ ((1-t^2)^{1/2}-1)\ P_i + t\ P_{ji}\,\big)\ g_l\peso{for every} l\in \I_k\ . 
$$
That is, if $V(t)=I+ ((1-t^2)^{1/2}-1)\ P_i + t\ P_{ji}\in\mat$ then $g_l(t)=V(t)\ g_l\,$ for every $l\in \I_k$ and $t\in[0,1)$.
Therefore, we get that
$$
\cG(t)=V(t)\,\cG=\{V(t)\, g_l\}_{l\in\I_k} \implies S_{\cG(t)}=V(t)\, S_\cG\, V(t)^*\peso{for} t\in[0,1)\ .
$$
Hence, we obtain the representation 
$$ 
S_{\cG(t)}=\sum_{\ell\in\I_d\setminus\{i,\,j\}}\mu_\ell \ v_\ell\otimes v_\ell 
+ \gamma_{11}(t)\ v_j\otimes v_j+ \gamma_{12}(t)\ v_j\otimes v_i 
+ \gamma_{21}(t)\ v_i\otimes v_j + \gamma_{22}(t)\ v_i\otimes v_i\ ,
$$ 
where the functions $\gamma_{rs}(t)$ are the entries of 
$A(t)=\big(\,\gamma_{rs}(t)\,\big)_{r\coma s=1}^2\in\cH(2)$ 
defined by 
$$
A(t)= \begin{pmatrix} 1 & t \\ 0 & (1-t^2)^{1/2}\end{pmatrix} \begin{pmatrix} \mu_j& 0 \\ 0&\mu_i\end{pmatrix} 
\begin{pmatrix}1& 0\\ t& (1-t^2)^{1/2}\end{pmatrix}  \peso{for every} t\in [0\coma 1)\ .
$$
It is straightforward to check that $\tr(A(t))=\mu_j+\mu_i$ and that $\det(A(t))=(1-t^2)\, \mu_j\,\mu_i\,$. 
These facts imply that 
if we consider the continuous function $L(t)=\la_{\max}(A(t))$ then $L(0)=\mu_j$ and $L(t)$ is strictly 
increasing in $[0,1)$. 
More straightforward computations show that we can consider continuous curves $x_i(t):[0,1)\rightarrow \C^2$ 
which satisfy that  $\{x_1(t),\,x_2(t)\}$ is ONB of $\C^2$ such that
$$A(t)\, x_1(t)=L(t)\, x_1(t) \peso{for} t\in[ 0,1) \py 
x_1(0)=e_1 \ ,\ \ x_2(0)=e_2\,.$$
For $t\in[0,1)$ we let $X(t)=(u_{r,s}(t))_{r,s=1}^2\in \cU(2)$ with columns $x_1(t) $ and $x_2(t)$. 
By construction, $X(t)=[0,1)\rightarrow \cU(2)$ 
is a continuous curve such that $X(0)=I_2\,$ and such that 
$$ X(t)^*\, A(t)\, X(t)=\begin{pmatrix} L(t) & 0 \\ 0 & \mu_i+\mu_j-L(t)\end{pmatrix}\,.$$
Finally, consider the continuous curve $U(t):[0,1)\rightarrow \matud$ given by 
$$
U(t)=  u_{11}(t) \,v_j\otimes v_j+u_{12}(t)\, v_j\otimes v_i + u_{21}(t)\, v_i\otimes v_j + u_{22}(t)\, v_i\otimes v_i
+\sum_{\ell\in\I_d\setminus\{i,\,j\}} v_l\otimes v_l \ .
$$ 
Notice that $U(0)=I$;  also, let $\tilde \cG(t)=U(t)^*\, \cG(t)\in\tcal $ for $t\in[0,1)$, which is a 
continuous curve such that $\tilde \cG(0)=\cG_0\,$.
 In this case, for $t\in[0,1)$ we have that 
$$
S_{\tilde \cG(t)}=U(t)^*\, S_{\cG(t)}\, U(t)
= L(t)\, v_j\otimes v_j + (\mu_i+\mu_j-L(t)) \,v_i\otimes v_i
+\sum_{\ell\in\I_d\setminus\{i,\,j\}}\mu_\ell \ v_\ell\otimes v_\ell \ .
$$ 
In other words, $U(t)$ is constructed in such a way that $\cB=\{v_l\}_{i\in\I_d}$ consists of eigenvectors of $S_{\tilde \cG(t)}$ for 
every $t\in[0,1)$. Hence, if 
$E(t) = L(t)-\mu_j\geq 0$ for $t\in [0\coma 1)$, we get that 
$$ 
S-S_{\tilde \cG(t)}= (c_r-E(t)\,)\, v_j\otimes v_j + (c_e+ E(t)\,)\, v_i\otimes v_i
+\sum_{\ell\in\I_d\setminus\{i,\,j\}}(\la_\ell - \mu_\ell) \ v_\ell\otimes v_\ell \ .
$$ 
Let $\varepsilon>0$ be such that 
$E(t)=L(t)-\mu_j \leq \frac{c_r-c_e}{2}$ for $t\in[0,\varepsilon]$. (recall that 
$L(0)=\mu_j $ and that $c_e<c_r$). %
Since $L(t)$ (and hence $E(t)$) is strictly 
increasing in $[0,1)$, we see that 
$$
(c_r- E(t)\coma c_e+ E(t))\prec (c_r\coma c_e)\implies \la(S-S_{\tilde \cG(t)})\prec \la(S-S_0) 
\peso{for} t\in(0\coma \varepsilon] \ , 
$$ 
where the majorization relations above are strict. 
\end{proof}

\pausa
\begin{proof}[Proof of Proposition \ref{nueva pro 3.10 y 3.11}]
Fix $S\in\matpos$,  $\asubi$  and a strictly convex u.i.n. $N$ on $\mat$. 
Consider $\cG_0$ a local minimizer of $\tnsa$ in $\tcal$. Then, $\cG_0$ satisfies
the assumptions in Notation \ref{nuevas notas 2}; with this notation, assume that $p>1$.
Then, we show that there exist $0=s_0<s_1<\ldots<s_{p-1}<s_p= \rk\, S_0\leq d$ such that 
\beq\label{los Kj3}
\barr{rl}
K_j &=J_j =\{s_{j-1}+1\coma \ldots \coma  s_j\} \ ,\quad 
\peso{for} j\in\IN{p-1}\  , \\&\\
K_p& =\{s_{p-1}+1 \coma \ldots \coma  s_p\} \ \ , \ \ J_p=\{s_{p-1}+1\coma \ldots\coma  k\} \ .
\earr 
\eeq
Indeed, in case the sets $J_j$ for $j\in\I_p$ do not have the structure described above (i.e. 
increasing sets formed by consecutive indexes) then, we get that there exist indexes $i,\, r\in\I_p$ and $h,\, l\in\I_k$ for which
Eq. \eqref{eq nueva pro a1} holds. In this case, Proposition  \ref{pro apend 1} shows that there exists a continuous curve $\cG(t):[0,1)\rightarrow \tcal$ such that $\cG(0)=\cG_0$ and such that 
$\la(S-S_{\cG(t)})\prec \la(S-S_0)$ with strict majorization for $t\in (0,\varepsilon)$ for some $\varepsilon >0$.
Since $N$ is a strictly convex u.i.n. we conclude that 
\beq \label{eq no es min loc apend}
\tnsa(\tilde \cG(t))=N(S-S_{\tilde \cG(t)})< N(S-S_0)
=\tnsa(\cG_0)\peso{for} t\in(0\coma \varepsilon]\ .
\eeq
This last fact contradicts the local minimality of $\cG_0$. Hence, there exist indexes  
$s_0=0<s_1<\ldots<s_{p-1}<s_p\leq d$ for which the representation of the sets $J_j$ for $j\in\I_p$ as in 
Eq. \eqref{los Kj3} holds. 

\pausa 
Similarly, in case $K_j$ for $j\in\I_p$ are not increasing sets formed by consecutive indexes then, using Proposition 
\ref{pro apend 2}, we also get that $\cG_0$ is not a local minimizer; this last fact contradicts the hypothesis on $\cG_0$.
Finally, notice that by Theorem \ref{teo applic1} we have that
 the family $\{g_i\}_{i\in J_j}$ is linearly independent for every $j\in\I_{p-1}$. 
In particular, by Eq. \eqref{cajas2}, we get that $\dim (W_j)=|K_j| =|J_j|  \peso{for} j\in\I_{p-1}$. Hence, 
we get that $J_j=K_j$ for $j\in\I_{p-1}$ and that $K_p=\{s_{p-1}+1 \coma \ldots \coma  s_p\}$ and the result follows.
\end{proof}

\pausa
In what follows, we show Theorem \ref{teo estruc min loc detallada}.  First, we consider a preliminary result. 

\begin{pro}\label{pro prom para c1}
Consider Notation \ref{nota los his} and \ref{nuevas notas 2}, and assume that $p>1$. Assume further that the sets $J_j$ and $K_j$, for $j\in\I_p$, satisfy Eq.
\eqref{los Kj3} above. Then, 
\ben
\item 
We have that  $(a_i)_{i\in J_j}\prec (\la_i-c_j)_{i\in K_j}$, for $j\in\I_p$.
\item If $0\leq r<s\leq d$ then, $(a_j)_{j=r+1}^s\prec(\la_j-P_{r+1,\, s})_{j=r+1}^s$ if and only if $$ P_{r+1,\,s}\leq P_{r,\, i}\ , \ \ r+1\leq i\leq s \iff P_{r+1,\, s}=\min\{P_{r+1,\, i}:\ r+1\leq i\leq s\}\,.$$
\een
\end{pro}
\begin{proof}
For each $j\in\I_p$, consider $W_j=\text{span}\{g_i:\ i\in J_j\}=R(S_{\cG_j})$, so that $\dim W_j=|K_j|$ and let $Q_j$ be the orthogonal projection onto $W_j$;
 then, $W_j$ reduces both $S$, $S_0$ and notice that $(S-S_0)\, Q_j= c_j\, Q_j$ and $S_0\, Q_j=S_{\cG_j}$. Then, 
$$S\, Q_j=(S-S_0)Q_j+ S_0\, Q_j = c_j\, Q_j + S_{\cG_j} \implies \la(S_{\cG_j})=((\la_i-c_j)_{i\in K_j},\,0_{d-|K_j|})\in (\R^d)\da\,. $$
Hence, by the Schur-Horn theorem we get that $(a_i)_{i\in J_j}\prec \la(S_{\cG_j})$ which is equivalent to the 
majorization relation $(a_i)_{i\in\I_j}\prec(\la_i-c_j)_{i\in K_j}$, and item 1 follows. 

\pausa 
Let $0\leq r<s\leq d$ and notice that by construction $(a_j)_{j=r+1}^s,\, (\la_j-P_{r+1\coma s})_{j=r+1}^s\in (\R^{s-r})\da$. 
On the other hand, if $r+1\leq i\leq s$ then
$$ \sum_{j=r+1}^i a_j \leq \sum_{j=r+1}^i \la_j-P_{r+1\coma s} \iff (i-r) P_{r+1\coma s}\leq \sum_{j=r+1}^i h_i \iff P_{r+1\coma s}\leq P_{r+1\coma i}\,. $$
This last fact shows item 2.
\end{proof}

	\begin{proof}[Proof of Theorem \ref{teo estruc min loc detallada}]
In case $\cG_0$ is a local minimizer of $\tnsa$ on $\tcal$ for a strictly convex u.i.n., then
the previous results imply that the sets $J_j$ and $K_j$ associated with $\cG_0$ satisfy Eq.
\eqref{los Kj3}. Hence, we show that the following relations hold:
\ben
\item The index $s_1 = \max \, \big\{j \le s_{p-1} \, :\, 
P_{1\coma j} = \min\limits_{i\le s_{p-1}}  \, P_{1\coma i} \, \big\}$, and 
$c_1 = P_{1\coma s_1}\,$.
\item Recursively, if $s_j<s_{p-1}\,$, then
$$
s_{j+1} = \max \, \big\{s_j< r \le s_{p-1} \, :\, 
P_{s_j+1\coma r} = \min\limits_{s_j< i\le s_{p-1}}  \, P_{s_j+1 \coma i} \, \big\} 
\py c_{j+1} = P_{s_j+1\coma s_{j+1}}\ .
$$
\een
Indeed, consider an arbitrary $0\leq j\leq p-2$. By item 1. in Proposition \ref{pro prom para c1} and the fact that $J_{j+1}=K_{j+1}=\{s_{j}+1,\ldots, s_{j+1}\}$ then we see that 
\beq\label{hay mayo apend}
(a_i)_{i\in J_{j+1}}\prec (\la_i-c_{j+1})_{i\in K_{j+1}} \implies c_{j+1}=P_{s_{j}+1\coma  s_{j+1}}\,.
\eeq
Now, using the majorization relation in Eq. \eqref{hay mayo apend} an item 2 in Proposition \ref{pro prom para c1} we also get that 
$$P_{s_j+1\coma s_{j+1}}=\min\{P_{s_j+1\coma i}: s_j< i\leq s_j \}\,. $$
Therefore, in case the relations between the indexes $s_0=0<\ldots<s_{p-1}$ and the constants $c_1<\ldots<c_{p-1}$ in the statement do not hold, we get that there 
exists $0\leq j\leq p-2$ such that 
$$
s_{j+1} < \max \, \big\{s_j< r \le s_{p-1} \, :\, 
P_{s_j+1\coma r} = \min\limits_{s_j< i\le s_{p-1}}  \, P_{s_j+1 \coma i} \, \big\} =t\leq s_{p-1}\,.
$$ By definition of $t$ we get that 
\beq \label{eq para contra apend}
c_{j+1}=P_{s_j+1 \coma s_{j+1}}\geq  P_{s_j+1 \coma t}\,.
\eeq Also, there exists
$ j+1\leq \ell\leq p-2$ such that $s_{\ell}< t\leq s_{\ell+1}$. Using the majorization relation in Eq. \eqref{hay mayo apend}
we see that for $j\leq r\leq \ell-1$: 
$$(s_{r+1}-s_{r}) \, c_{r+1} = \sum_{i=s_{r}+1}^{s_{r+1}}h_i \py   (t-s_{\ell}) \, c_{\ell+1} \leq  \sum_{i=s_{\ell}+1}^{t} h_i\,.  $$
Then, the previous inequalities allow us to bound
$$
 P_{s_j+1\coma t}=\frac{1}{t-s_j} \ \sum_{i=s_j+1}^t  h_i\geq \sum_{r=j}^{\ell-1}   \frac{s_{r+1}-s_{r}}{t} \  c_{r+1} + \frac{t-s_{\ell}}{t} \  c_{\ell+1} =:\beta
$$ that represents the lower bound $\beta$ as a convex combination of the constants $c_{j+1}<\ldots <c_{\ell+1}$. 
This last fact clearly implies that $P_{s_j+1\coma t}\geq \beta > c_{j+1}$, that contradicts Eq. \eqref{eq para contra apend}.
\end{proof}


\begin{thebibliography}{}

{\scriptsize

\bibitem{Illi} J. Antezana, P. Massey, M. Ruiz and D. Stojanoff, 
The Schur-Horn theorem for operators and frames with prescribed norms and frame operator, 
 Illinois J.  Math., { 51} (2007),  537-560.
 

\bibitem{BF} J.J. Benedetto, M. Fickus,  Finite normalized tight frames, Adv. Comput. Math. 18, No. 2-4 (2003), 357-385 .


\bibitem{Bhat} R. Bhatia,  Matrix Analysis,
Berlin-Heildelberg-New York, Springer 1997.

\bibitem{BoC10} B.G. Bodmann,  P.G: Casazza, The road to equal-norm Parseval frames. J. Funct. Anal. 258 (2010), no. 2, 397-420.

\bibitem{CC13} J. Cahill,  P.G. Casazza, The Paulsen problem in operator theory. Oper. Matrices 7 (2013), no. 1, 117-130. 

\bibitem{CFM12} J. Cahill, M. Fickus, D.G. Mixon, Auto-tuning unit norm frames. Appl. Comput. Harmon. Anal. 32 (2012), no. 1, 1-15.


\bibitem{TaF} P.G. Casazza, The art of frame theory, Taiwanese J. Math. 4 (2000), no. 2, 129-201.



\bibitem{FinFram} P. G. Casazza and G. Kutyniok eds., Finite Frames: Theory and Applications. Birkhauser, 2012. xii + 483 pp.




\bibitem{Chr} O. Christensen, An introduction to frames and Riesz bases. Applied and Numerical Harmonic Analysis. Birkhäuser Boston, 2003. xxii+440 pp.

\bibitem{CMR} G. Corach, P. Massey and M. Ruiz, Procrustes problems and Parseval quasi-dual
frames, Acta Applicandae Mathematicae 131 (2014), 179–195


\bibitem{Eld} L. Eld\'en, H. Park, A Procrustes problem on the Stiefel manifold.Numer. Math. 82,
no. 4, 599-619 (1999)


\bibitem{GB} J. C. Gower,  G. B. Dijksterhuis,  Procrustes problems. Oxford Statistical Science
Series, 30. Oxford University Press, Oxford, (2004)

\bibitem{Hig} N.J. Higham,  Matrix nearness problems and applications. Applications of matrix
theory (Bradford, 1988), 1-27, Inst. Math. Appl. Conf. Ser. New Ser., 22, Oxford
Univ. Press, New York, (1989)

\bibitem{Ka} T. Kato, Perturbation Theory for Linear Operators, Springer-Verlag, Berlin Heidelberg New York 1980.

\bibitem{Ki} U. Kintzel,  Procrustes problems in finite dimensional indefinite scalar product
spaces. Linear Algebra Appl. 402, 1-28 (2005)


\bibitem {LPS}
C.K. Li, Y.T. Poon, T. Schulte-Herbrüggen, 
Least-squares approximation by elements from matrix orbits achieved by gradient flows on compact Lie groups. Math. Comp. 80 (2011), no. 275, 1601-1621. 

\bibitem {dnp}
P. Massey, N. B. Rios, D. Stojanoff; Frame completions with prescribed norms: local minimizers and applications. Adv. in Comput. Math. 44 (2018), no. 1, 54-86.

\bibitem {dnp2}
P. Massey, N. B. Rios, D. Stojanoff; Local Liskii's theorems for unitarily invariant norms.  Linear Algebra Appl. 557 (2018), 34-61. 

\bibitem{MR0}P. Massey, M.A. Ruiz, Tight frame completions with prescribed norms. Sampl. Theory Signal Image Process. 7 (2008), no. 1, 1-13.


\bibitem{mr2010}
P. Massey, M. Ruiz; Minimization of convex functionals over frame operators. Adv. Comput. Math. 32 (2010), no. 2, 131-153.



\bibitem {mrs2}
P.G. Massey, M.A. Ruiz, D. Stojanoff; Optimal frame completions.
Advances in Computational Mathematics 40 (2014), 1011-1042.



\bibitem {strawn}
N. Strawn; Optimization over finite frame varieties and structured dictionary design, Appl. Comput. Harmon. Anal. 32 (2012) 413-434.

\bibitem {strawn2} N. Strawn; Finite frame varieties: nonsingular points, tangent spaces, and explicit local parameterizations. J. Fourier Anal. Appl. 17 (2011), no. 5, 821-853.

\bibitem{Wat} G. A. Watson,  The solution of orthogonal Procrustes problems for a family of orthogonally invariant norms. Adv. Comput. Math. 2 , no. 4, 393-405 (1994)

}
\end{thebibliography}
\end{document}